\documentclass[twoside, 11pt]{article}

\usepackage{enumerate}
\usepackage{mathrsfs}
\usepackage{amsmath, color}
\usepackage{latexsym}
\usepackage{amssymb}
\usepackage{amsthm}
\usepackage{amsfonts}
\usepackage{dsfont}
\usepackage{fancyhdr}
\usepackage{hyperref}
\usepackage{xcolor,cancel}
\usepackage[switch,columnwise]{lineno}
\usepackage{lipsum}

\newtheorem{dfn}{Definition}[section]
\newtheorem{thm}[dfn]{Theorem}
\newtheorem{rmk}[dfn]{Remark}

\newtheorem{prop}[dfn]{Proposition}

\newtheorem{lem}[dfn]{Lemma}
\theoremstyle{definition}

\numberwithin{equation}{section}

\newcommand{\IE}{{\mathbb{E}}}
\newcommand{\IP}{{\mathbb{P}}}
\newcommand{\IR}{{\mathbb{R}}}
\newcommand{\FF}{{\mathcal{F}}}
\newcommand{\EE}{{\mathcal{E}}}

\newcommand{\IZ}{{\mathbb{Z}}}

\def\eps{\varepsilon}
\def\wh{\widehat}
\def\wt{\widetilde}
\def\<{\langle}
\def\>{\rangle}

\title{\bf  Discrete Approximation to Brownian Motion with Varying Dimension in Bounded Domains}

\date{\today}
\author{{\bf Shuwen Lou}}

\begin{document}

\maketitle

\begin{abstract}
In this paper we study a discrete approximation to Brownian motion with varying dimension (BMVD in abbreviation) introduced in \cite{CL}  by continuous time random walks on square lattices.  The state space of BMVD  contains a $2$-dimensional component, a $3$-dimensional component, and a ``darning point" which joins these two components.  Such a state space is  equipped with the geodesic distance, under which BMVD is a diffusion process.  In this paper, we  prove that  BMVD restricted on a bounded domain containing the darning point is the weak limit of continuous time reversible random walks with exponential holding times. 

\end{abstract}



\section{Introduction}\label{Intro}

Brownian motion on spaces with varying dimension has been introduced and studied in details in \cite{CL}.   Such a process is an interesting example of Brownian motion on non-smooth spaces  and can be characterized nicely via Dirichlet form. The state space of BMVD  looks like a plane with a vertical half line installed on it. Roughly speaking, it is ``embedded" in the following space:
\begin{equation*}
\IR^2 \cup \IR_+ =\{(x_1, x_2, x_3)\in \IR^3: x_3=0 \textrm{ or } x_1=x_2=0 \hbox{ and } x_3>0\}.
\end{equation*}
Here $\IR_+:=\{x\in \IR:   x>0\}$  is the set of positive real numbers.  As has been noted in \cite{CL}, Brownian motion cannot be defined on such a state space in the usual sense because a two-dimensional Brownian motion does not hit a singleton. Hence to construct BMVD,  in \cite{CL}, a closed disc on $\IR^2$ is ``shorted" to a singleton. In other words, the media offers zero resistance on this closed disc, so that  the process travels across it at infinite velocity. The resulting Brownian motion hits the shorted disc in finite time with probability one. Then we install an infinite pole at this ``shorted" disc.  

To be more precise, the state space of BMVD on $E$ is defined as follows.
Fix $\eps>0$ and denote by  $B_\eps $  the closed disk on $\IR^2$ centered at $(0,0)$ with radius $\eps $. Let 
${D_\eps}:=\IR^2\setminus  B_\eps $. By identifying $B_\eps $ with a singleton denoted by $a^*$, we  introduce a topological space $E:={D_\eps}\cup \{a^*\}\cup \IR_+$, with the origin of $\IR_+$ identified with $a^*$ and
a neighborhood of $a^*$ defined as $\{a^*\}\cup \left(V_1\cap \IR_+ \right)\cup \left(V_2\cap {D_\eps}\right)$ for some neighborhood $V_1$ of $0$ in $\IR^1$ and $V_2$ of $B_\eps $ in $\IR^2$. Let $m$ be the measure on $E$ whose restriction on $\IR_+$ or ${D_\eps}$ is $1$- or $2$-dimensional Lebesgue measure, respectively.  In particular,   we set   $m (\{a^*\} )=0$. Note that the measure $m$  depends on $\eps$, the radius of the ``hole" $B_\eps$.  The following definition for BMVD can be found in \cite[Definition 1.1]{CL}.
\begin{dfn}[Brownian motion with varying dimension]\label{def-bmvd}An $m$-symmetric diffusion process satisfying the following properties is called Brownian motion with varying dimension. 
\begin{description}
\item{\rm (i)} its part process in $\IR_+$ or $D_\eps$ has the same law as standard Brownian motion in $\IR_+$ or $D_\eps$;
\item{\rm (ii)} it admits no killings on $a^*$;
\end{description}
\end{dfn}
 It follows from the definition that BMVD spends zero amount of time under Lebesgue measure (i.e. zero sojourn time) at $a^*$. The following theorem is a restatement of   \cite[Theorem 2.2]{CL}, which  asserts that given  $\eps>0$,  BMVD exists and is unique in law. It also gives the Dirichlet form characterization for such a process.

\begin{thm}\label{BMVD-non-drift}
For every  $\eps >0$,
BMVD  on $E$ with parameter $\eps$ exists and is unique.
Its associated Dirichlet form $(\EE, \mathcal{D}(\EE))$ on $L^2(E; m)$ is given by
\begin{eqnarray*}
\mathcal{D}(\EE) &= &  \left\{f: f|_{D_\eps}\in W^{1,2}(D_\eps),  \, f|_{\IR_+}\in W^{1,2}(\IR_+),
\hbox{ and }
f (x) =f (0) \hbox{ q.e. on } {\partial D_\eps}\right\},  
\\
\EE(f,g) &=& \frac{1}{4} \int_{D_\eps}\nabla f(x) \cdot \nabla g(x) dx+\frac{1}{2}\int_{\IR_+}f'(x)g'(x)dx . 
\end{eqnarray*}
Furthermore, such a BMVD is a Feller process with strong Feller property. 
\end{thm}

\begin{rmk}
For the computation convenience in this paper, we let  BMVD in Theorem \ref{BMVD-non-drift} corresponds to the BMVD defined in \cite[Theorem 2.2]{CL} with parameters $(\eps, p=1)$ but running on $D_\eps$ at a speed $1/2$. 
\end{rmk}

Throughout  this paper,  the state space $E$ is equipped with the geodesic distance  $\rho$.  Namely, for $ x,y\in E$, $\rho (x, y)$ is the shortest path distance (induced
from the Euclidean space) in $E$ between $x$ and $y$.
For notation simplicity, we write $|x|_\rho$ for $\rho (x, a^*)$.
We use $| \cdot |$ to denote the usual Euclidean norm. For example, for $x,y\in {D_\eps}$,
$|x-y| $ is  the Euclidean distance between $x$ and $y$ in $\IR^2$.
Note that for $x\in {D_\eps}$, $|x|_\rho =|x|-\eps$.
Clearly,
\begin{equation}\label{e:1.1}
\rho (x, y)=|x-y|\wedge \left( |x|_\rho + |y|_\rho \right)
\quad \hbox{for } x,y\in {D_\eps}
\end{equation}
and $\rho (x, y)= |x|+|y|-\eps$ when $x\in \IR_+$ and $y\in {D_\eps}$ or vice versa.
Here and in the rest of this paper, for $a,b\in \IR$, $a\wedge b:=\min\{a,b\}$.


It is well-known that Brownian motion on Euclidean spaces is the scaling limit of simple random walks on square lattices.  It was shown in \cite{BC} that reflected Brownian motion is the scaling limit of simple random walks on square lattices. This motivates us to study the discrete appromation to BMVD which is an example of strong Markov process constructed via ``darning".   Since not  until \cite{CL} was the concept of BMVD introduced, the amount of existing literature on this topic is very limited. The ``darning" point in the state space is very singular in the sense that it is not only nonsmooth, but also it is where two spaces with different dimensionalities are joined together. In this paper, we answer the question that from the perspective of discrete random walks, how one can ``interprete" or ``visualize"  Brownian motions on such an unsual space characterized in an abstract way via Dirichlet forms.

Towards this goal, we use the approach in \cite{BC}, but one of the major differences between this paper and \cite{BC} is that, rigorously speaking, the state space of BMVD is not Euclidean, thus the metric on the state space of BMVD under which BMVD is continuous cannot be the Euclidean distance but needs to be the geodesic distance $\rho$. Therefore, we need to establish the (C-)tightness and weak convergence of the random walks with respect to the geodesic distance $\rho$. One of the important ingredients in this paper is to find the comparability between $\rho$ and Euclidean distance on the state space $E$ near $a^*$. Furthermore, since BMVD is constructed by ``darning" together two Brownian motions with different dimensionalities using Dirichlet form theory, we need to find an explicit description to the behavior of the sequence of discrete random walks near the ``darning point", so that the limiting behavior of these random walks indeed converge to the BMVD.

To introduce the state space of  random walks with varying dimension, for $k\in \mathbb{N}$,  let $D_\eps ^k:=D_\eps \cap 2^{-k}\IZ^2$. Similar to the definition of the darning point $a^*$ in \cite{CL},  we identify vertices  of $2^{-k}\IZ^2$ that are contained in the closed disc $B_\eps$ as a singleton $a^*_k$. Let $E^k:=2^{-k}\IZ_+\cup \{a^*_k\}\cup D_\eps^k$.

Before introducing   the graph structure on $E^k$, in general we know a graph $G$ is defined by its vertices and edges, thus can be written as ``$G=\{G_v, G_e\}$", where $G_v$ is its collection of vertices, and $G_e$ is its connection of edges. Given any two vertices in $a,b\in G$,  if there is unoriented edge with endpoints $a$ and $b$,  we say {\it $a$ and $b$ are adjacent to each other in $G$} and write it as  `` $a\leftrightarrow b$ in $G$".   In this paper, it is always assumed that given two vertices $a, b$ on a graph, there is at most  one such unoriented edge connecting these two points (otherwise edges with same endpoints can be removed and replaced with one single edge). This unoriented edge is denoted by $e_{ab}$ or $e_{ba}$ ($e_{ab}$ and $e_{ba}$ are viewed as the same elelment in $G_e$).  For notation convenience, in this article we denote by $\mathcal{G}_2:=\{2^{-k}\IZ^2, \mathcal{V}_2\}$ the $2$-dimensional square lattice over $2^{-k}\IZ^2$, where $\mathcal{V}_2$ is the collection of the edges of this graph. We also denote by $\mathcal{G}_1:=\{2^{-k}\IZ^1_+\cup \{0\}, \mathcal{V}_1\}$ the  $1$-dimensional  square lattice over $2^{-k}\IZ_+\cup \{0\}$, where $\mathcal{V}_1$ is its collection of edges  and $\IZ_+:=\{l\in \IZ:l>0\}$ is the set of positive integers . 

Now we introduce the graph structure on $E^k$.  For this, we identify $\{0\}$ with $\{a^*_k\}$.    We let $G^k=\{G^k_v, G^k_e\}$ be a graph where $G^k_v=E^k$ is the collection of vertices and $G^k_e$ is the collection of unoriented edges over $E^k$  defined as follows:
\begin{align*}
G^k_e:=&\{e_{xy}:\, \exists \,x,y\in D_\eps^k, |\,x-y|=2^{-k},\, e_{xy}\in\mathcal{V}_2,\, e_{xy}\cap B_\eps=\emptyset\}
\\
\cup &\{e_{xy}: \exists \, x,y\in 2^{-k}\IZ_+\cup \{0\}, \,|x-y|=2^{-k},\, e_{xy}\in 
\mathcal{V}_1\}
\\
\cup &\{e_{xa^*_k}:x\in  D^k_\eps,\text{ there is at least one }y\in  2^{-k}\IZ^2\cap B_\eps \text{ with }|x-y|=2^{-k}, \,e_{xy}\in \mathcal{V}_2\}.
\end{align*}
  Note that $G^k=\{G^k_v, G^k_e\}$ is a connected graph. We emphasize that  given any $x\in G^k_v$, $x\neq a^*_k$, there is at most one element in $G^k_e$ with endpoints $x$ and $a^*_k$.   We denote by $v_k(x)=\#\{e_{xy}\in G^k_e\}$, i.e., the number of vertices in $G^k_v$ adjacent to $x$.  $E^k$ is equipped with the following underlying reference measure:
\begin{equation}\label{def-mk}
m_k(x):=\left\{
    \begin{aligned}
         &\frac{2^{-2k}}{4}v_k(x),  &x \in D^k_\eps;\\
        &\frac{2^{-k}}{2}v_k(x), &x \in 2^{-k}\IZ_+;\\
        & \frac{2^{-k}}{2}+\frac{2^{-2k}}{4}\left(v_k(x)-1\right), & x=a^*_k.
    \end{aligned}
\right.
\end{equation}
To give definition to a random walk on bounded spaces with varying dimension, we consider the following  Dirichlet form on $L^2(E^k, m_k)$:
\begin{align}\label{DF-RWVD-form}
\left\{
\begin{aligned}
&\mathcal{D}(\EE^{k})=L^2(E^k, m_k)
\\
&\EE^{k}(f, f)= \frac{1}{8}\sum_{\substack{e^o_{xy}:\; e_{xy}\in G^k_e,\\ x,y\in D^k_\eps\cup \{a_k^*\} }} \left(f(x)-f(y)\right)^2 +\frac{2^k}{4}\sum_{\substack{e^o_{xy}:\;e_{xy}\in G^k_e,\\ x,y\in 2^{-k}\IZ_+\cup \{a_k^*\} }}\left(f(x)-f(y)\right)^2,
\end{aligned}
\right.
\end{align}
where $e^o_{xy}$ is an {\it oriented edge} from  $x $  to  $y$. In other words, given any pair of  adjacent  vertices $x, y \in G^k_v$,  the edge with endpoints $x$ and $y$ is represented twice in the sum: $e^o_{xy}$ and $e^o_{yx}$.  It is easy to verify that $(\EE^k, \mathcal{D}(\EE^{k}))$ on $L^2(E^k, m_k)$ is a regular symmetric Dirichlet form, therefore there is a symmetric strong Markov process associated with it. We denote this process by $X^k$.
The explicit distribution of $X^k$ is given in Proposition \ref{jump-distribution-Xk}. To describe it in words,  $X^k$ stays at each vertex of $E^k$ for an exponentially distributed holding time with parameter $2^{2k}$ before jumping to one of its neighbors in $E^k$.  Starting from every vertex except $a^*_k$, upon every move, $X^k$ jumps to any of its nearest neighbors with equal probability.  From  $a^*_k$, at every move,  $X^k$ jumps to its neighbor in $2^{-k}\IZ_+$ with probability $2^{k+1}\cdot [v_k(a_k^*)+2^{k+1}-1]^{-1}$, while to each of its neighbors in $D_\eps^k$ with probability $[v_k(a^*_k)+2^{k+1}-1]^{-1}$.

In this paper, we study random walk approximation to BMVD in bounded domains.  We let  $E_0\subset E$ be an open bounded domain containing $a^*$ satisfying the following conditions: 
\begin{description}\label{assumption}
\item[Assumption] $E_0$ is a bounded Lipschitz domain. 
\end{description}
Let $E^k_0:= E_0\cap E^k$.    Similar to the definition for graph $G^k$, we set the graph $G_0^k=\{G^k_{0,v}, G^k_{0,e}\}$, where $G^k_{0,v}=E^k_0$ is the collection of vertices of $G^k_0$  and 
\begin{align*}
G^k_{0,e}:=&\{e_{xy}\subset E_0:\,\exists \,x,y\in D_\eps^k\cap E_0, \,|x-y|=2^{-k}, \,e_{xy}\in \mathcal{V}_2, \, e_{xy}\cap B_\eps=\emptyset\}
\\
\cup &\{e_{xy}\subset E_0: \exists \,x,y\in (2^{-k}\IZ_+\cap E_0)\cup \{0\}, \,|x-y|=2^{-k}, \, e_{xy}\in \mathcal{V}_1\}
\\
\cup &\{e_{xa^*_k}\subset E_0: x\in D^k_\eps\cap E_0,\text{ there is at least one }y\in  2^{-k}\IZ^2\cap B_\eps \text{ with }|x-y|=2^{-k}, e_{xy}\in \mathcal{V}_2\}
\end{align*}
is the collection of unoriented edges of $G^k_0$.   In view of the assumption that $E_0$ is connected, for sufficiently large $k
\in \mathbb{N}$,  $G_0^k$ is also a connected graph.  Here for notational  simplicity, without loss of generality, we assume that for all $k\ge 1$,  $G_0^k$ is also a connected graph. We denote by $\overline{v}_k(x)=\#\{e_{xy}\in G^k_{0, e}\}$, i.e., the number of vertices in $G^k_{0, v}$ adjacent to $x$.   We equip $E_0^k$ with the following measure $\overline{m}_k$:
\begin{equation}\label{def-bar-mk}
\overline{m}_k(x):=\left\{
    \begin{aligned}
         &\frac{2^{-2k}}{4}\bar{v}_k(x),  &x \in D^k_\eps;\\
        &\frac{2^{-k}}{2}\bar{v}_k(x), &x \in 2^{-k}\IZ_+;\\
        & \frac{2^{-k}}{2}+\frac{2^{-2k}}{4}\left(\bar{v}_k(x)-1\right), & x=a^*_k,
    \end{aligned}
\right.
\end{equation}
where $\bar{v}_k(x)$ is the number of vertices in $E^0_k$ that are adjacent to $x$. Here we emphasis that for $x\in E^k_0$, since $\bar{v}_k(x)$ is not necessarily equal to $v_k(x)$, $\overline{m}_k(x)$ may differ from $m_k(x)$. The boundary of $E_0^k$ is defined as
\begin{equation*}
\partial E^k_0 :=\{x\in E_0^k:\, \exists\; y\in E^k\backslash E_0^k \text{ such that }x\leftrightarrow y \text{ in }G^k\}.
\end{equation*}
It is thus easy to see that 
\begin{equation}\label{boundary-lattice-domain}
\partial E^k_0=\left\{ x\in E^k\cap D_\eps:\, \overline{v}_k(x)<4, \, v_{k}(x)=4 \right\}\cup \left\{ x\in E^k\cap 2^{-k}\IZ_+:\; \overline{v}_k(x)<2,\, v_k(x)=2 \right\}.
\end{equation}

Without loss of generality, throughout this paper we assume  $\rho(a^*, \, \partial E_0)>16\eps$, and that $2^{-k}< \eps/4$ for all $k\ge 1$.  Denote the part process of $X^k$ killed upon hitting $\partial E_0^k$ by $\wh{X}^{k}$.  Our main result is the following theorem.  
\begin{thm}\label{main-result}
For every $T>0$, the laws of  $\{\wh{X}^{k}, \IP^{\overline{m}_k}\}_{k\ge 1}$ are tight in the space $\mathbf{D}([0, T], \overline{E}_0, \rho)$ equipped with Skorokhod topology. Furthermore, $(\wh{X}^{k}, \IP^{\overline{m}_k}) $ converges weakly to the part process of BMVD with parameter $\eps$ killed upon leaving $E_0$ as $k\rightarrow \infty$. 
\end{thm}

For technical convenience, we will often consider stochastic processes whose initial distribution is a finite measure, not necessarily normalized to have total mass $1$, for example, the Lebesgue measure on a bounded set. Translating the results to the usual probabilistic setting is straightforward  and therefore skipped.

We close this section by giving an outline of the rest of this article as well as the ideas of the proof.  In Section \ref{S:2} we give a quick review on continuous-time reversible Markov chains and their associated Dirichlet forms.  The proof to the main result Theorem \ref{main-result} is presented in Section \ref{S:3}. To prove Theorem \ref{main-result}, we first consider $\{\wt{X}^k, \IP^{\overline{m}_k}; k\ge 1\}$,  a class of random walks  on $E_0$ ``reflected" upon hitting $\partial E_0$ whose rigorous definition is given in Proposition \ref{jump-distribution}. Using the property that $\{\wt{X}^k,\IP^{\overline{m}_k}; k\ge 1\}$ have infinite lifetimes, we show that they are C-tight in law. Using Dirichlet form theory, it can further be shown that  $\{\wt{X}^k,\IP^{\overline{m}_k}; k\ge 1\}$ is a resurrected process of  $\{\wh{X}^k, \IP^{\overline{m}_k};k\ge 1\}$. Therefore one can equivalently construct  $\{\wt{X}^k,\IP^{\overline{m}_k}; k\ge 1\}$ from  $\{\wh{X}^k, \IP^{\overline{m}_k};k\ge 1\}$  using the  Ikeda-Nagasawa-Watanabe ``piecing-together" procedure.  This implies that $\{\wh{X}^k, \IP^{\overline{m}_k};k\ge 1\}$, the part processes of $\{X^k\}_{k\ge 1}$ killed upon exiting $E_0$ are also C-tight in law. From here, using the martingale problem method, we show that the  limit of any weakly convergent subsequence of $\{\wt{X}^k, \IP^{\overline{m}_k};\; k\ge 1\}$ restricted on $E_0$ is a BMVD. Finally, with a bit more argument, this leads to  Theorem \ref{main-result}.

\subsection{Preliminaries:  continuous-time reversible pure jump processes and symmetric Dirichlet forms}\label{S:2}

In this section, we give a brief background on continuous-time reversible pure jump processes   and symmetric Dirichlet forms. The results in this section can be found in \cite[\S 2.2.1]{CF}.

Suppose $\mathsf{E}$ is a locally compact separable metric space and $\{Q(x, dy)\}$ is a probability kernel on $(\mathsf{E}, \mathcal{B}(\mathsf{E}))$ with $Q(x, \{x\})=0$ for every $x\in \mathsf{E}$. Given a constant $\lambda >0$, we can construct a pure jump Markov process $\mathsf{X}$ as follows: Starting from $x_0\in \mathsf{E}$, $\mathsf{X}$ remains at $x_0$ for an exponentially distributed holding time $T_1$ with parameter $\lambda(x_0)$ (i.e., $\IE[T_1]=1/\lambda(x_0)$), then it jumps to some $x_1\in \mathsf{E}$ according to distribution $Q(x_0, dy)$; it remains at $x_1$ for another exponentially distributed holding time $T_2$  also with  parameter $\lambda(x_1)$ before jumping to $x_2$ according to distribution $Q(x_1, dy)$.  $T_2$ is independent of $T_1$. $\mathsf{X}$ then continues. The probability kernel $Q(x, dy)$ is called the {\it road map} of $\mathsf{X}$, and the $\lambda(x)$ is its {\it speed function}. If there is a $\sigma$-finite measure $\mathsf{m}_0$ on $\mathsf{E}$ with supp$[\mathsf{m}_0]=\mathsf{E}$ such that
\begin{equation}\label{symmetrizing-meas}
Q(x, dy)\mathsf{m}_0(dx)=Q(y, dx)\mathsf{m}_0(dy),
\end{equation}
$\mathsf{m}_0$ is called a {\it symmetrizing measure} of the road map $Q$. Another way to view \eqref{symmetrizing-meas} is that, $Q(x, dy)$ is the one-step transition ``probability" distribution, so its density with respect to the symmetrizing measure $Q(x,dy)/\mathsf{m}_0(dy)$ must be symmetric in $x$ and $y$, i.e., 
\begin{equation*}
\frac{Q(x, dy)}{\mathsf{m}_0(dy)}=\frac{Q(y, dx)}{\mathsf{m}_0(dx)}. 
\end{equation*}
The following theorem is a restatement of  \cite[Theorem 2.2.2]{CF}.
\begin{thm}[\cite{CF}]\label{DF-pure-jump}
Given a speed function $\lambda >0$. Suppose \eqref{symmetrizing-meas} holds, then the reversible pure jump process $\mathsf{X}$ described above can be characterized by the following Dirichlet form $(\mathfrak{E}, \mathfrak{F})$ on $L^2(\mathsf{E}, \mathsf{m}(x))$ where the underlying reference measure is $\mathsf{m}(dx)=\lambda(x)^{-1}\mathsf{m}_0(dx)$ and
\begin{equation}\label{DF-EN}
\left\{
    \begin{aligned}
        &\mathfrak{F}=  L^2(\mathsf{E},\; \mathsf{m}(x)), \\
        &\mathfrak{E}(f,g) = \frac{1}{2} \int_{\mathsf{E}\times \mathsf{E}} (f(x)-f(y))(g(x)-g(y))Q(x, dy)\mathsf{m}_0(dx).
    \end{aligned}
\right.
\end{equation}
\end{thm}

\section{Continuous-time random walks on lattices with varying dimension} \label{S:3}

To begin with, we show that the measures $m_k$ and $\overline{m}_k$ both converge to Lebesgue measure restricted on $E_0$. 
\begin{prop} \label{P:2.1}
On $E_0$, as $k\rightarrow \infty$,
\begin{equation*}
m_k\stackrel{\text{weakly}}{\Rightarrow } m|_{E_0}, \quad \text{and} \quad \overline{m}_k\stackrel{\text{weakly}}{\Rightarrow } m|_{E_0}.
\end{equation*}
\end{prop}

\begin{proof}
It suffices to prove the first result. In particular, it suffices to show that 
\begin{equation*}
m_k(a^*_k)\rightarrow 0, \quad \text{as }k\rightarrow \infty. 
\end{equation*}
 Towards this, in the following we first claim that for all $k\ge 1$, 
\begin{equation}\label{P2.1-1}
\#\{x:\; x\leftrightarrow a^*_k\text{ in }G^k,\, x\in D^k_\eps\}\le 56\eps\cdot 2^{k}+28.
\end{equation}
If a vertex in $D_\eps^k$  is adjacent to $a_k^*$, then it must be within $2\cdot 2^{-k}$ Euclidean distance from $\partial B_\eps$. Also, any pair of  vertices in $\{x:\; x\leftrightarrow a_k^*  \text{ in }G^k, \, x\in D^k_\eps\}$ must be  at least $2^{-k}$ distances apart. For $x\in \IR^2$, we denote by 
\begin{equation*}
B(x, r):=\{y\in \IR^2: \; |y-x|\le r\}.
\end{equation*}
 Thus the family of balls
 \begin{equation*}
 \left\{ B(x, 2^{-k}/2):\,  x\leftrightarrow a_k^* \text{ in } G^k \right\}
\end{equation*} 
 are disjoint. Notice that in Section 1  it  is assumed that $2^{-k}<\eps/4$, thus  all such balls must be contained in $B(0, \eps+4\cdot 2^{-k})\backslash B(0, \eps-3\cdot 2^{-k})$, which has an area of $\pi(14\eps\cdot 2^{-k}+7\cdot 2^{-2k})$. This implies that there can at most be 
\begin{equation*}
\frac{\pi\left(14\eps\cdot 2^{-k}+7\cdot 2^{-2k}\right)}{\pi \left(\frac{1}{2}\cdot 2^{-k}\right)^2}=56\eps \cdot 2^k+28
\end{equation*}
many vertices  in the set $\{x:\; x\leftrightarrow a_k^*\text{ in }G^k\}$. This verifies \eqref{P2.1-1}.   Therefore,
\begin{align*}
m_k(a^*_k)&=\frac{2^{-k}}{2}+\frac{2^{-2k}}{4}\left(v_k(a^*_k)-1\right)\le \frac{2^{-k}}{2}+\frac{2^{-2k}}{4}\cdot \left(56\eps \cdot 2^k+27\right)\rightarrow 0, \, \text{as }k\rightarrow \infty. 
\end{align*}
This proves the statement.
\end{proof}

 Using Theorem \ref{DF-pure-jump}, in the following proposition we  give a description to the behavior  of $X^k$ in the unbounded space with varying dimension $E^k$.

\begin{prop}\label{jump-distribution-Xk}
For every $k=1,2,\dots$, $X^k$ has  constant  speed function $\lambda_k=2^{2k}$ and  road map
\begin{equation*}
J_k(x, dy)=\sum_{z\in E^k, \;z\leftrightarrow x \text{ in }G^k}  j_k(x,z)\delta_{\{z\}}(dy), 
\end{equation*}
where
\begin{description}
\item{(i)} 
\begin{equation}\label{original-p_k(x,y)-1}
j_k(x, y)= \frac{1}{v_k(x)}, \quad \text{if }x\in D^k_\eps\cup 2^{-k}\IZ_+, \,y\leftrightarrow x  \text{ in }G^k
\end{equation}
\item{(ii)}
\begin{equation}\label{original-p_k(x,y)-2}
j_k(a^*_k, y)=\left\{
    \begin{aligned}
         &\frac{1}{v_k(a^*_k)+2^{k+1}-1},  &y \in D^k_\eps,\, y\leftrightarrow x \text{ in }G^k;\\
        &\frac{2^{k+1}}{v_k(a_k^*)+2^{k+1}-1}, &y \in 2^{-k}\IZ_+,\, y\leftrightarrow x \text{ in }G^k.
    \end{aligned}
\right.
\end{equation}
\end{description}
\end{prop}

\begin{proof}
Given \eqref{DF-RWVD-form}, it is immediate that
$(\EE^k, \mathcal{D}(\EE^{k}))$  on $L^2(E^k, m_k)$ is  a regular symmetric Dirichlet form. We define a measure $m_{0, k}:=\lambda_km_k(x)$ on $E^k$. In view of the definition of $m_k$ given in \eqref{def-mk},  $m_{0, k}$ can be expressed as
\begin{equation}\label{e:2.6}
m_{0,k}(x)=\left\{
    \begin{aligned}
         &\frac{v_k(x)}{4},  &x \in D^k_\eps;\\
        &\frac{2^k}{2}v_k(x), &x \in 2^{-k}\IZ_+;\\
        & \frac{2^{k}}{2}+\frac{1}{4}\left(v_k(x)-1\right), & x=a^*_k,
    \end{aligned}
\right.
\end{equation}
We note that in $G^k$,  $a^*_k$ has exactly one neighbor in $2^{-k}\IZ_+$ and thus $v_k(a_k^*)-1$ neighbors in $D^k_\eps$. Therefore, from \eqref{e:2.6} one may directly verify that for the $J_k$ a defined in the statement of the proposition,  it satisfies for any adjacent pair of $x, y\in E^k$ that
\begin{equation}\label{e:2.5}
 J_k(x, dy) m_{0,k}(x)= j_k(x, y) m_{0,k}(x)=\left\{
    \begin{aligned}
         &\frac{1}{4},  &x,y \in D^k_\eps\cup \{a^*_k\}, \, x\leftrightarrow y  \text{ in }G^k ;\\
        &\frac{2^{k}}{2}, &x,y \in 2^{-k}\IZ_+\cup \{a^*_k\},\, x\leftrightarrow y  \text{ in }G^k .
    \end{aligned}
\right.
\end{equation}
The desired conclusion now readily follows from \eqref{DF-RWVD-form}  and   Theorem \ref{DF-pure-jump}.
\end{proof}

To rigorously  introduce random walks in a  bounded domain $E_0^k$, we first consider  the following regular symmetric Dirichlet form on $L^2(E_0^k, \overline{m}_k)$, where $
\overline{m}_k$ is defined in \eqref{def-bar-mk}: 
\begin{align}\label{reflected-form}
\left\{
\begin{aligned}
&\mathcal{D}(\wt{\EE}^{k})= L^2(E_0^k, \overline{m}_k),
\\
&\wt{\EE}^{k}(f, f)= \frac{1}{8}\sum_{\substack{e^o_{xy}:\;e_{xy}\in G^k_{0, e},\\ x,y\in D^k_\eps\cup \{a_k^*\} }}\left(f(x)-f(y)\right)^2 +\frac{2^k}{4}\sum_{\substack{e^o_{xy}:\;e_{xy}\in G^k_{0, e},\\ x,y\in 2^{-k}\IZ_+\cup \{a_k^*\} }}\left(f(x)-f(y)\right)^2.
\end{aligned}
\right.
\end{align}
In the next proposition we  give the road map  of $\wt{X}^k$, which implies that $\wt{X}^k$ behaves like ``reflected" simple random walk on  $\partial E^k_0$.  Furthermore, one can infer from the following proposition that $\wt{X}^k$ is conservative, i.e., being of an infinite lifetime.

\begin{prop}\label{jump-distribution}
$(\wt{\EE}^k, \mathcal{D}(\wt{\EE}^{k}))$   is a regular symmetric Dirichlet form on $L^2(E_0^k, \overline{m}_k)$.    Therefore there is an $\overline{m}_k$-symmetric pure jump process on $E_0^k$ associated with it.   Denote this process by $\wt{X}^k$.  $\wt{X}^k$ travels with constant   speed $\lambda_k=2^{2k}$ and has  road map
\begin{equation*}
\wt{J}_k(x, dy)=\sum_{x\in E_0^k, z\leftrightarrow x \text{ in  } G_0^k}\wt{j}_k(x,z)\delta_{\{z\}}(dy), 
\end{equation*}
where
\begin{description}
\item{(i)} 
\begin{equation}\label{p_k(x,y)-1}
\wt{j}_k(x, y)= \frac{1}{\bar{v}_k(x)}, \quad \text{if }x\in D^k_\eps\cup 2^{-k}\IZ_+, \,y\leftrightarrow x \text{ in }G^k_0.
\end{equation}
\item{(ii)}
\begin{equation}\label{p_k(x,y)-2}
\wt{j}_k(a^*_k, y)=\left\{
    \begin{aligned}
         &\frac{1}{\bar{v}_k(a^*_k)+2^{k+1}-1},  &y \in D^k_\eps,\, y\leftrightarrow x \text{ in }G^k_0;\\
        &\frac{2^{k+1}}{\bar{v}_k^*(a_k^*)+2^{k+1}-1}, &y \in 2^{-k}\IZ_+,\, y\leftrightarrow x \text{ in }G^k_0.
    \end{aligned}
\right.
\end{equation}
\end{description}
\end{prop}
\begin{proof}
The proof is  identical to that of Proposition \ref{jump-distribution-Xk}, therefore omitted. 
\end{proof}

To describe the ``coordinates" in $E^k$ as a subspace embedded in $\IR^3$, we consider the following projection maps $f_i: E^k\rightarrow \IR_+$:
\begin{equation}\label{def-f2}
f_1(x)=\left\{
    \begin{aligned}
         & |x|_\rho\cdot\frac{x_1}{|x|},  &\text{if } x\in D^k_\eps;\\
        & 0, & \text{otherwise};
    \end{aligned}
\right.
\end{equation}
and 
\begin{equation}\label{def-f3}
f_2(x)=\left\{
    \begin{aligned}
         &  |x|_\rho \cdot \frac{x_2}{|x|},  &\text{if } x\in D^k_\eps;\\
        & 0, & \text{otherwise};
    \end{aligned}
\right.
\end{equation}
and 
\begin{equation}\label{def-f1}
f_3(x)=\left\{
    \begin{aligned}
         & |x|,  &\text{if } x\in 2^{-k}\IZ_+;\\
        & 0, & \text{otherwise};
    \end{aligned}
\right.
\end{equation}

Before we prove the tightness of $\wt{X}^k$, we first record the following inequality which is proved using elemetary geometry which  will be used later in this Section. 
\begin{lem}\label{P:3.4}
For $i=1,2,3$, it holds for every pair of $x,y$ in $D^k_\eps$ such that $x\leftrightarrow y $   in $G^k$
\begin{equation}\label{eq:P3.4}
|f_i(x)-f_i(y)|\le 9\cdot 2^{-k}.
\end{equation}
\end{lem}
\begin{proof}
The  inequality is trivial when $i=3$. For $i=1$ or $2$, without loss of generality, we prove the result for $i=1$. Recall that it is assumed in \S 1 that $2^{-k}< \eps/4$. By elementary algebra,
\begin{align*}
\left|f_1(x)-f_1(y)\right|&=\left|\frac{(|x|-\eps)}{|x|}\cdot x_1-\frac{(|y|-\eps)}{|y|}\cdot y_1\right|
\\
&=\left|x_1-y_1-\eps \left(\frac{x_1}{|x|}-\frac{y_1}{|y|}\right)\right|
\\
&\le \left|x_1-y_1\right|+\eps \left|\frac{x_1}{|x|}-\frac{y_1}{|y|}\right|
\\
&\le \left|x_1-y_1\right|+\eps\left|\frac{x_1}{|x|}-\frac{x_1}{|y|}+\frac{x_2-y_2}{|y|}\right|
\\
&\le \left|x_1-y_1\right|+\frac{\eps}{|y|}\cdot \left|x_1-y_1\right|+\eps\cdot |x_1|\cdot \left|\frac{1}{|x|}-\frac{1}{|y|}\right|
\\
&\le 2|x_1-y_1|+\eps\cdot |x_1|\cdot \frac{\left||x|-|y|\right|}{|x|\cdot |y|}
\\
&\le  2\cdot 2^{-k}+\eps\cdot |x_1|\cdot \frac{2^{-k}}{|x|\cdot |y|} \le 3\cdot 2^{-k},
\end{align*}
where the second last inequality is due to the fact that  the Euclidean norm $|y|\ge \eps$. The proof is thus complete. 
\end{proof}

Roughly speaking, the following lemma states that  for any $x, y\in E$, if the Euclidean distance between $(f_1(x), f_2(x), f_3(x))$ and $(f_1(y), f_2(y), f_3(y))$ is small, then so is $\rho(x,y)$. The proof is elementary but technical. 
\begin{lem}\label{P:3.5}
There exist constants $C_1$ and $C_2$ such that
\begin{equation}\label{e:5.3}
|x|_\rho\le C_1\left|\left(f_1(x), f_2(x), f_3(x)\right)\right|, \quad \text{for all  }x\in E^k,
\end{equation}
and 
\begin{equation}\label{e:5.4}
\rho(x, y)\wedge \rho(x,y)^2\le C_2\left|\left(f_1(x)-f_1(y), f_2(x)-f_2(y), f_3(x)-f_3(y)\right)\right|, \quad \text{for }x,y\in E^k.
\end{equation}
\end{lem}
\begin{proof}
To prove   \eqref{e:5.3}, we consider the two cases that $x\in \IR_+$ and $x\in D_\eps$. \eqref{e:5.3} is obviously true when $x\in \IR_+$. In this proof we denote the Euclidean coordinates of $x\in D_\eps$  by $(x_1, x_2, 0)$.    When $x\in D_\eps$,
\begin{equation*}
\left|\left(f_1(x), f_2(x), f_3(x)\right)\right|=\left(|x|_\rho^2\cdot \frac{x_1^2}{|x|^2}+|x|_\rho^2\cdot \frac{x_2^2}{|x|^2}+0^2\right)^{1/2}=|x|_\rho.
\end{equation*}
To show \eqref{e:5.4},  we first note that when at least one of $x$ and $y$ is in $\IR_+$, it holds 
\begin{equation*}
\rho(x, y)=\left|\left(f_1(x)-f_1(y), f_2(x)-f_2(y), f_3(x)-f_3(y)\right)\right|.
\end{equation*}
When both  $x, y \in D_\eps$,  we express the Euclidean coordinates of $x$ and $y$ in $\IR^3$ by $(r_x\cos\alpha, r_x\sin\alpha, 0)$ and $(r_y\cos \beta, r_y\sin\beta, 0)$, respectively. We note that on $D_\eps$, it holds 
\begin{equation*}
\rho(x,y)=\left(|x|_\rho+|y|_\rho\right)\wedge |x-y|.
\end{equation*}
Therefore, using the elementary formula for the distance between two points  in terms of polar coordinates, 
\begin{equation}\label{expression-rho-xy}
\rho(x,y)=\left[(r_x-\eps)+(r_y-\eps)\right]\wedge \sqrt{r_x^2+r_y^2-2r_xr_y\cos(\beta-\alpha)}.
\end{equation}
For the right hand side of \eqref{e:5.4}, we have
\begin{equation*}
\left(f_1(x), f_2(x), f_3(x)\right)=\left(\frac{r_x-\eps}{r_x}x_1, \frac{r_x-\eps}{r_x}x_2, 0\right)= \left((r_x-\eps)\cos\alpha, (r_x-\eps)\sin\alpha, 0\right)
\end{equation*}
and
\begin{equation*}
 \left(f_1(y), f_2(y), f_3(y)\right)=\left(\frac{r_y-\eps}{r_y}y_1, \frac{r_y-\eps}{r_y}y_2, 0\right)=\left((r_y-\eps)\cos\beta, (r_y-\eps)\sin\beta, 0\right).
\end{equation*}
Again it follows from the distance formula in polar coordinates that
\begin{eqnarray}
&&\left|\left(f_1(x)-f_1(y), f_2(x)-f_2(y), f_3(x)-f_3(y)\right)\right| \nonumber
\\
&=&\sqrt{(r_x-\eps)^2+(r_y-\eps)^2-2(r_x-\eps)(r_y-\eps)\cos(\beta-\alpha).}\label{e:3.4}
\end{eqnarray}
From here, we further divide our discussion into three cases depending on $\cos(\beta -\alpha)$.
\\
{\it Case 1. } $\cos(\beta-\alpha)\le 0$. It can easily be seen that
\begin{eqnarray}
&&\left|\left(f_1(x)-f_1(y), f_2(x)-f_2(y), f_3(x)-f_3(y)\right)\right| \nonumber
\\
&\ge & \sqrt{(r_x-\eps)^2+(r_y-\eps)^2} \ge \frac{1}{2}\left[(r_x-\eps)+(r_y-\eps)\right]\ge \frac{1}{2}\rho(x,y). \label{e:3.5}
\end{eqnarray}
{\it Case 2. }  $0<\cos(\beta-\alpha) \le 3/4$. In this case, 
\begin{eqnarray}
&&\left|\left(f_1(x)-f_1(y), f_2(x)-f_2(y), f_3(x)-f_3(y)\right)\right| \nonumber
\\
&=&\sqrt{(r_x-\eps)^2+(r_y-\eps)^2-2(r_x-\eps)(r_y-\eps)\cos(\beta-\alpha)} \nonumber
\\
&=& \sqrt{\cos(\beta-\alpha)\cdot (r_x-r_y)^2+(1-\cos(\beta -\alpha))\left[(r_x-\eps)^2+(r_y-\eps)^2\right]} \nonumber
\\
&\ge & \sqrt{1-\cos(\beta -\alpha)}\cdot \sqrt{(r_x-\eps)^2+(r_y-\eps)^2} \nonumber
\\
\eqref{expression-rho-xy}&\ge& \frac{1}{4}\cdot \frac{1}{2}\left[(r_x-\eps)+(r_y-\eps)\right]\ge \frac{1}{8} \rho(x, y). \label{e:3.6}
\end{eqnarray}
{\it Case 3. }  $\cos(\beta-\alpha) > 3/4$. In view of \eqref{expression-rho-xy}, we have
\begin{align}
\rho(x,y)^2&=\left[r_x^2+r_y^2-2r_xr_y\cos(\beta-\alpha)\right]\wedge \left[(r_x-\eps)+(r_y-\eps)\right]^2 \nonumber
\\
&=\left[\cos(\beta-\alpha)\cdot (r_x-r_y)^2+(1-\cos(\beta-\alpha))(r_x^2+r_y^2)\right]\wedge \left[(r_x-\eps)+(r_y-\eps)\right]^2  \label{e:3.7}
\end{align}
From  \eqref{e:3.4}, we have 
\begin{eqnarray}
&&\left|\left(f_1(x)-f_1(y), f_2(x)-f_2(y), f_3(x)-f_3(y)\right)\right| \nonumber
\\
&=& \sqrt{\cos(\beta-\alpha)(r_x-r_y)^2+(1-\cos(\beta -\alpha))\left[(r_x-\eps)^2+(r_y-\eps)^2\right]}  \label{e:3.8}
\end{eqnarray}
To proceed, we further divide our discussion into two subcases. 
\\
{\it Subcase 1. } $r_x+r_y \ge 8\eps$. In this case, $\max\{r_x, r_y\}\ge 4\eps$. Without loss of generality, we assume $r_x\ge r_y$, which implies that $r_x\ge 4\eps$.  Then it is easy to see
\begin{equation}\label{e:3.14}
\frac{1}{8}\left(r_x^2+r_y^2\right)\le \frac{1}{4}r_x^2\le (r_x-\eps)^2\le  (r_x-\eps)^2+ (r_y-\eps)^2.
\end{equation}
Therefore,
\begin{align*}
\rho(x,y)^2&\stackrel{\eqref{e:3.7}}{\le} \left[\cos(\beta-\alpha)\cdot (r_x-r_y)^2+(1-\cos(\beta-\alpha))(r_x^2+r_y^2)\right]
\\
&\stackrel{\eqref{e:3.14}}{\le} 8 \bigg[\cos(\beta-\alpha)\cdot (r_x-r_y)^2+\Big(1-\cos(\beta-\alpha)\Big)\Big((r_x-\eps)^2+(r_y-\eps)^2\Big)\bigg]
\\
&\stackrel{\eqref{e:3.8}}{=} 8 \left|\left(f_1(x)-f_1(y), f_2(x)-f_2(y), f_3(x)-f_3(y)\right)\right|^2.
\end{align*}
{\it Subcase 2.}   $r_x+r_y \le 8\eps$. In this case, $\eps\le r_x, r_y\le 7\eps$. Thus  it holds
\begin{equation}\label{e:3.15}
 r_x^2+r_y^2\le 2\cdot 7^2\eps^2, \quad (r_x-r_y)^2\le 7^2\eps^2,\quad \text{and } (r_x-\eps)^2+(r_y-\eps)^2\le 2\cdot 7^2\eps^2.
\end{equation}
Using the elementary inequality 
\begin{equation}\label{elementary-ineq-1}
xy\ge \left(x^2 \wedge y^2\right),\quad \text{if }x,y>0,
\end{equation}
we have that there exists some $c_1>0$ (whose exact value below may differ from line to line)  such that
\begin{eqnarray}
&&\left|\left(f_1(x)-f_1(y), f_2(x)-f_2(y), f_3(x)-f_3(y)\right)\right|^2  \nonumber
\\
&\stackrel{\eqref{e:3.8}}{=}&   \cos(\beta-\alpha)(r_x-r_y)^2+(1-\cos(\beta -\alpha))\left[(r_x-\eps)^2+(r_y-\eps)^2\right]   \nonumber
\\
 &\stackrel{\eqref{elementary-ineq-1}}{\ge}&  \cos(\beta-\alpha)(r_x-r_y)^2+ \left(1-\cos(\beta -\alpha)\right)^2\wedge \left[(r_x-\eps)^2+(r_y-\eps)^2\right]^2\nonumber
 \\
 &\stackrel{\eqref{e:3.15}}{\ge} &  \cos(\beta-\alpha)(r_x-r_y)^2+ \left[\frac{1}{\left(2\cdot 49\eps^2\right)^2}\left(1-\cos(\beta -\alpha)\right)^2\left(r_x^2+r_y^2\right)^2\right]\wedge \left[(r_x-\eps)^2+(r_y-\eps)^2\right]^2\nonumber
 \\
 &\stackrel{\eqref{e:3.15}}{\ge}  &c_1\bigg[\left(\cos(\beta-\alpha)\left(r_x-r_y\right)^2\right)^2\wedge\left[(r_x-\eps)^2+(r_y-\eps)^2\right]^2 \nonumber
 \\
 &+& \left(1-\cos(\beta -\alpha)\right)^2\left(r_x^2+r_y^2\right)^2\wedge \left[(r_x-\eps)^2+(r_y-\eps)^2\right]^2\bigg]\nonumber
 \\
 &\ge &  c_1 \cdot \max\left\{\left(\cos(\beta-\alpha)(r_x-r_y)^2\right)^2, \left(1-\cos(\beta -\alpha)\right)^2\left(r_x^2+r_y^2\right)^2 \right\}   \wedge  \left[(r_x-\eps)^2+(r_y-\eps)^2\right]^2\nonumber
 \\
 &\ge & \frac{c_1}{4} \cdot \left[\left(\cos(\beta-\alpha)(r_x-r_y)^2\right)^2+\left(1-\cos(\beta -\alpha)\right)^2\left(r_x^2+r_y^2\right)^2 \right]\wedge  \left[(r_x-\eps)^2+(r_y-\eps)^2\right]^2\nonumber
 \\
 &\stackrel{\eqref{e:3.7}}{\ge }& c_1\cdot \rho (x, y)^4.
\end{eqnarray}
Combining the two subcases above, we conclude that when $\cos(\beta-\alpha)>3/4$, there exists some $c_2>0$ such that
\begin{equation*}
\left|\left(f_1(x)-f_1(y), f_2(x)-f_2(y), f_3(x)-f_3(y)\right)\right|^2 \ge  c_2\min\left\{\rho(x,y), \, \rho(x, y)^2\right\},
\end{equation*}
which proves the desired result for Case 3. The proof is complete by combining Cases 1-3.
\end{proof}

In the next lemma, before establishing the tightness of the laws of $(\wh{X}^k)_{k\ge 1}$,  we  first show the tightness of $\{\wt{X}^k, \IP^{\overline{m}_k};k\ge 1\}$.   We note that unlike $\wh{X}^k$, $\wt{X}^k$ has infinite lifetime, therefore $\overline{m}_k$ is the  invariant measure of $\wt{X}^k$.

\begin{lem}\label{L:3.4}
For every $T>0$, the laws of $\{\wt{X}^k, \IP^{\overline{m}_k}, k\ge 1\}$ are C-tight in the space $\mathbf{D}([0, T], \overline{E}_0, \rho)$ equipped with the Skorokhod topology. 
\end{lem}

\begin{proof}
For every  continuous $f: \overline{E}_0\rightarrow \IR$ such that $f|_{\IR_+\cap \overline{E}_0}\in C^1(\IR_+\cap \overline{E}_0)$ and $f|_{D_\eps\cap \overline{E}_0}\in C^1(D_\eps\cap \overline{E}_0)$,
 we denote by $M^{k, f}$ the martingale additive functional of finite energy in the Fukushima decomposition of $f(\wt{X}^k_t)-f(\wt{X}^k_0)$ (see, i.e., \cite[Theorem 4.2.6]{CF}).  From Theorem \ref{DF-pure-jump}, one can see that a L\'{e}vy system $(\wt{N}_k, \wt{H}_k)$ of  $\wt{X}^k$ can be chosen as 
 \begin{equation*}
 \wt{N}_k(x, dy)=\wt{J}_k(x, dy),\quad \wt{H}_{k, t}=\lambda_k t=2^{2k}t.
 \end{equation*}
In the following we apply \cite[(4.3.9)]{CF} to express the quadratic variation of the martingale part appearing in the Fukushima decomposition   of $f(\wt{X}^k_t)-f(\wt{X}^k_0)$  in terms of a L\'{e}vy system. Indeed, by Proposition \ref{jump-distribution}, we have 
\begin{align*}
\langle M^{k, f}\rangle_t &=\int_0^t  \int_E  \left(f(\wt{X}_s^k)-f(y)\right)^2 \wt{N}_k(\wt{X}_s^k, dy)2^{2k}ds
\\
&=\int_0^t \sum_{y\leftrightarrow \wt{X}_s^k  \text{ in }G^k_0 }\left(f(\wt{X}_s^k)-f(y)\right)^2 \wt{j}_k(\wt{X}_s^k, y)2^{2k}ds.
\end{align*}
Recall that $f_i$, $i=1,2,3$ defined in \eqref{def-f1}-\eqref{def-f3} are continuous on $\overline{E}_0$ and of $C^1$-class on $E_0\backslash \{a^*\}$. Below we first claim that the laws of $\{\sum_{i=1}^3 \langle M^{k, f_i}\rangle_t, k\ge 1\}$ are C-tight.
Lemma \ref{P:3.4}  implies that for $i=1,2,3$, 
\begin{equation}\label{e:2.24}
|f_i(x)-f_i(y)|\le 9\cdot 2^{-k}, \quad \text{for }x,y\in D^k_\eps, \, x\leftrightarrow y \text{ in }G^k_0.
\end{equation}
Thus for all $k\ge 1$, $i=1,2,3$, $t>s$,
\begin{eqnarray*}
\langle M^{k, f_i}\rangle_t -\langle M^{k, f_i}\rangle_s  &=& \int_s ^t \sum_{y\leftrightarrow \wt{X}^k_u \text{ in }G^k_0  } \left( f_i(\wt{X}_u^k)-f_i(y)\right)^2 \wt{j}_k(\wt{X}_u^k, y) 2^{2k} du
\\
&
\stackrel{\eqref{e:2.24}}{\le} & \int_s^t 9^2\cdot  2^{-2k}\cdot 2^{2k} \sum_{y\leftrightarrow \wt{X}^k_u \text{ in }G^k_0  }\wt{j}_k(\wt{X}_u^k, y) \;du =81\cdot (t-s).
\end{eqnarray*}
Hence by \cite[Proposition VI.3.26]{JS}, the laws of the sequence $
\{\sum_{i=1}^3 \langle M^{k, f_i}\rangle_t, k\ge 1\}$ are C-tight in $\mathbf{D}([0, T], \IR, |\cdot|)$. Then by \cite[Theorem VI.4.13]{JS},  the laws of  the $\IR^3$-valued processes $\{(M^{k, f_1}, M^{k, f_2}, M^{k, f_3}), k\ge 1\})$ are  tight in $\mathbf{D}([0, T], \IR^3, |\cdot|)$.  

For every $k\ge 1$, without loss of generality we may assume that $\Omega$ is the canonical space $\mathbf{D}([0, \infty), \overline{E}_0,\rho)$, and $\wt{X}^k_t$ is the coordinate map  on $\Omega$. 
For any time  $t>0$, let $r_t$ be the time-reversal operator from time $t$ for $\wt{X}^k$:
\begin{equation}\label{def-time-reversal}
r_t(\omega)(s)=\left\{
    \begin{aligned}
         &\omega((t-s)^-),  &\text{if } 0\le s\le t;\\
        &\omega (0), &\text{if } s\ge t.
    \end{aligned}
\right.
\end{equation}
It is obvious that $\wt{X}^k$ is stochastically continuous, therefore $\wt{X}^k$ is continuous at  every time $t>$, $\IP^{\overline{m}_k}$-a.s. Also, for any $T>1$, since each $\wt{X}^k$ has an infinite lifetime, i.e., $\zeta= \infty$ a.s., $\IP^{\overline{m}_k}$  restricted on the time interval $[0, T)$   is invariant under $r_T$. By Fukushima decomposition,   for $i=1,2,3$, it holds that
\begin{eqnarray}
f_i(\wt{X}^k_t)-f_i(\wt{X}^k_0) &=& M_t^{k, f_i}+N^{k, f_i}_t \nonumber
\\
&=& \frac{1}{2}M_t^{k, f_i}+\frac{1}{2}M_t^{k, f_i} +N^{k, f_i}_t \nonumber
\\
&=& \frac{1}{2}M_t^{k, f_i}+\left(\frac{1}{2}M_t^{k, f_i} +N^{k, f_i}_t\right)\circ r_{T+1}\circ r_{T+1} \nonumber
\\
&=& \frac{1}{2}M_t^{k, f_i} -\frac{1}{2}\left(M_{(T+1)-}^{k, f_i} - M_{(T+1-t)-}^{k, f_i}\right)\circ r_{T+1}, \,t\in [0, T), \label{eq3.4}
\end{eqnarray}
where the last equality above is due to \cite[p.\;955, lines 12-14]{CFKZ}. Next we prove that for every $T>0$ and every $i=1,2,3$,  the right hand side of \eqref{eq3.4} is tight in the space $\mathbf{D}([0, T+1), \IR)$. Towards this,  for notational convenience, for every pair of  $\theta, \delta>0$ we define the following sets in $\Omega$:
\begin{equation*}
A^i_{1, \theta, \delta}:=\left\{\omega\in \Omega:\;  \inf_{\substack{0=t_0<t_1<\cdots <t_r=T+1,\\ \inf_{1\le i<r}(t_i-t_{i-1})\ge \theta}} \max_{i\le r} \sup_{s, t\in [t_{i-1}, t_i)} \left| f_i(\wt{X}^k_t(\omega))- f_i (\wt{X}^k_s(\omega)) \right|\ge \delta \right\},
\end{equation*}
\begin{equation*}
A^i_{2, \theta, \delta}:=\left\{\omega\in \Omega:\;  \inf_{\substack{0=t_0<t_1<\cdots <t_r=T+1,\\ \inf_{1\le i<r}(t_i-t_{i-1})\ge \theta}} \max_{i\le r} \sup_{s, t\in [t_{i-1}, t_i)} \left| \frac{1}{2}M_t^{k, f_i}(\omega) - \frac{1}{2}M^{k, f_i}_s(\omega) \right|\ge \delta \right\},
\end{equation*}
and 
\begin{equation*}
A^i_{3, \theta,\delta}:=\left\{\omega\in \Omega:\;  \inf_{\substack{0=t_0<t_1<\cdots <t_r=T+1,\\ \inf_{1\le i<r}(t_i-t_{i-1})\ge \theta}} \max_{i\le r} \sup_{s, t\in [t_{i-1}, t_i)} \left| \mathring{M}_t^{k, f_i}(\omega) - \mathring{M}^{k, f_i}_s(\omega) \right|\ge \delta \right\},
\end{equation*}
where $\mathring{M}_t^{k, f_i}(\omega):=\frac{1}{2}\left(M_{(T+1)-}^{k, f_i} - M_{(T+1-t)-}^{k, f_i}\right)\circ r_{T+1}(\omega)$. It then follows that 
\begin{equation*}
A^i_{1, \theta, \delta}\subset  \left( A^i_{2, \theta, \frac{\delta}{2}} \cup  A^i_{3, \theta, \frac{\delta}{2}}\right),\,\text{and } \, \IP^{\overline{m}_k}\left( A^i_{2, \theta, \frac{\delta}{2}} \right)=\IP^{\overline{m}_k}\left( A^i_{3, \theta, \frac{\delta}{2}} \right),
\end{equation*} 
where the equality above results from the definition of $r_{T+1}$, as well as the fact that $\IP^{\overline{m}_k}$  restricted on the time interval $[0, T+1)$   is invariant under $r_{T+1}$. This further implies that 
\begin{equation}\label{L2.6-compute-1}
\IP^{\overline{m}_k}\left( A^i_{1, \theta, \delta} \right)\le 2\IP^{\overline{m}_k}\left( A^i_{2, \theta, \frac{\delta}{2}} \right), \quad \text{for all }\theta, \delta>0.
\end{equation}
Using   \eqref{L2.6-compute-1}, by  verifying the two conditions in  \cite[Theorem VI 3.21]{JS},  one can easily conclude  from the tightness of  $M^{k, f_i}$ that  each $\{f_i(\wt{X}^k),k\ge 1\}$, $i=1,2,3$,  is tight  in $\mathbf{D}([0, T+1), \IR, |\cdot|)$.  Hence the laws of
\begin{equation*}
\left\{\left(f_1(\wt{X}^k_t), f_2(\wt{X}^k_t),f_3(\wt{X}^k_t)\right),\,\IP^{\overline{m}_k}\right\}_{k\ge 1}
\end{equation*}
are tight in  $\mathbf{D}([0, T], \IR^3, |\cdot|)$, for any $T>0$.   Next we use Lemma \ref{P:3.5} to show the tightness of $(\wt{X}^k)_{k\ge 1}$. As a standard notation, for $x\in \mathbf{D}([0, +\infty), \IR^3, d)$  where $d(\cdot,\, \cdot)$ is a metric on $\IR^3$, we denote by
\begin{equation}\label{def-dubiu}
w_d(x,\, \theta,\, T):=\inf_{\{t_i\}} \max_{i} \sup_{s, t\in [t_i, t_{i-1}]} d(x(s), x(t)),
\end{equation}
where $\{t_i\}$ ranges over all possible partitions of the form $0=t_0<t_1<\cdots <t_{n-1}<T\le t_n$ with $\min_{1\le i\le n} (t_i-t_{i-1})\ge \theta$ and $n\ge 1$. Again in view of \cite[Theorem VI 3.21]{JS}, the tightness of $\{(f_1(X^k_t), f_2(X^k_t),f_3(X^k_t)), \, k\ge 1\}$ is equivalent to the conjunction of:
\begin{description}
\item{{\rm (i)}} For any $T>0$, $\delta>0$, there exist $k_0\in \mathbb{N}$ and $K>0$ such that for all $k\ge k_0$,
\begin{equation}\label{e:3.12}
\IP^{\overline{m}_k}\left[\sup_{t\in [0, T]}\left|(f_1\wt{X}^k_t), f_2(\wt{X}^k_t),f_3(\wt{X}^k_t)\right|>K\right]<\delta. 
\end{equation} 
\item{{\rm (ii)}} For any $T>0$, $\delta_1, \delta_2>0$, there exist $\theta>0$ and $k_0>0$ such that for all $k\ge k_0$,
\begin{equation}\label{e:3.13}
\IP^{\overline{m}_k}\left[\sup_{t\in [0, T]}w_{|\cdot|}\left( (f_1(\wt{X}^k), f_2(\wt{X}^k), f_3(\wt{X}^k))  ,\, \theta,\, T\right)>\delta_1\right]<\delta_2. 
\end{equation}
\end{description}
Finally, on account of Lemma \ref{P:3.5}, \eqref{e:3.12} and \eqref{e:3.13} together  imply:
\begin{description}
\item{{\rm (i)}} For any $T>0$, $\delta>0$, there exist $k_0\in \mathbb{N}$ and $K>0$ such that for all $k\ge k_0$,
\begin{equation}\label{e:3.12-1}
\IP^{\overline{m}_k}\left[ \sup_{t\in [0, T]} \big|\wt{X}^k_t\big|_\rho>N\right]<\delta. 
\end{equation} 
\item{{\rm (ii)}} For any $T>0$, $\delta_1, \delta_2>0$, there exist $\theta>0$ and $k_0>0$ such that for all $k\ge k_0$,
\begin{equation}\label{e:3.13-1}
\IP^{\overline{m}_k}\left[w_\rho\left( \wt{X}^k  ,\, \theta,\, T\right)>\delta_1\right]<\delta_2. 
\end{equation}
\end{description}
This proves the tightness of $\{\wt{X}^k,\IP^{\overline{m}_k}; k\ge 1\}$  in the space $\mathbf{D}([0, T],\overline{E}_0, \rho)$. Finally, to conclude the C-tightness of $\{\wt{X}^k,\IP^{\overline{m}_k}; k\ge 1\}$ (i.e., tight with all subsequential limits supported on the set of continuous paths), it suffices to note  \cite[Chapter 3, Theorem 10.2]{EK}.
\end{proof}

Before  proving the next lemma to establish the uniform convergence of the generators of $\wt{X}^k$, for notational  convenience, we define the following class of functions:
\begin{align}
\mathcal{G}:&=\{f:\IR^2\cup \IR_+\rightarrow \IR,\,f|_{B_\eps}=\text{const}=f|_{\IR_+}(0),\, f|_{\IR^2}\in C^3(\IR^2), f|_{\IR_+} \in C^3(\IR_+), \nonumber
\\
&f \text{ is supported on a compact subset of }(E_0\backslash \{a^*\})\cup B_\eps \}.\label{def-class-G}
\end{align}
Every $f\in \mathcal{G}$ can be uniquely  identified as a function mapping $E$ to $\IR$. Thus for $f\in \mathcal{G}$,  we define
\begin{equation}\label{def-wt-Lk}
\mathcal{\wt{L}}_k f(x):=2^{2k}\sum_{ \substack{y\in E^k_0, \\y\leftrightarrow x \text{    in }G^k_0}} \left(f(y)-f(x)\right)\wt{J}_k (x, dy), \quad \text{ for }x\in E^k_0.
\end{equation}
For the rest of this paper we set 
\begin{equation}\label{def-S0k}
S^{k}_0:=\{x\in D_\eps^k\cap E_0^k:\; \bar{v}_{k}(x)=4)\}\cup \{x\in 2^{-k}\IZ_+\cap E^k_0:\; \bar{v}_{k}(x)=2\}.
\end{equation}
 It is easy to see that $\{S^k_0\}_{k\ge 1}$ is an increasing sequence of sets. Also, from \eqref {boundary-lattice-domain}, one can see that if a vertex $x\in( \IR_+\cup D^k_\eps)\backslash S^k_0$,  then either $x\leftrightarrow a^*_k$ in $G^k_0$, or $x\in \partial E_0^k$. 
\begin{lem}\label{Lemma2.6}
 For every $\delta>0$ and every $f\in \mathcal{G}$, there exists some $k_{\delta, f}\in \mathbb{N}$, such that for all $k\ge k_{\delta, f}$:
 \begin{description}
\item{(i)} \, $\overline{m}_k(E_0^k\backslash S_0^k)<\delta$;
 \item{(ii)}  As $k\rightarrow \infty$,  $\mathcal{\wt{L}}_kf$ converges uniformly to 
\begin{equation}\label{def-L}
\mathcal{L}f:=\frac{1}{2}\Delta f|_{\IR_+}+\frac{1}{4}\Delta f|_{D_\eps} \quad \text{on } S_0^{k_{\delta, f}}.
\end{equation}
 \end{description}
Also, there exists some constant $C_1>0$ independent of $k$ such  that for all $k\ge 1$ and all $x\in E^k_0$, 
\begin{equation*}
\wt{\mathcal{L}}_k f(x)\le C_1. 
\end{equation*}
\end{lem}
\begin{proof}
Given any $\delta>0$, we know there exists some $k_1\in \mathbb{N}$ such that for all $k\ge k_1$, $\overline{m}_k(E_0^k\backslash S_0^k)<\delta$. Also, 
for $f\in \mathcal{G}$, since $f$ is supported on  a compact subset  of $E_0$,  there must  exist some $k_2\in \mathbb{N}$ such that for all $k\ge k_2$,  $f$ is constant on $E^k_0\backslash  (S^{k}_0\cup \{x: x\leftrightarrow a_k^* \text{ or }x=a^*_k\})$. Below we claim that it suffices to take $k_{\delta ,f}:=\max\{k_1, k_2\}$.

To show the uniform convergence of  $\mathcal{\wt{L}}_kf (x)$ on $S^{k_{\delta, f}}_0$, we divide our discussion into two cases depending on the position of $x$.  Noticing that $f\in \mathcal{G}$ is defined on $\IR^2\cup \IR_+$, in the following we are allowed to use Taylor expansion in terms of Euclidean coordinates. 
\\
{\it Case 1. } $x\in 2^{-k}\IZ_+\cap S^k_0$.  By Taylor expansion, we have
\begin{align}\label{taylor-1}
\mathcal{\wt{L}}_{k} f(x)&=2^{2k}\sum_{y\leftrightarrow x \text{ in }G^k_0} \left[f'(x)\cdot (y-x)+\frac{1}{2}f''(x)\cdot (y-x)^2 +O(1)|y-x|^3 \right]\wt{j}_{k}(x, y).
\end{align} 
{\it Case 2. } $x\in D_\eps^k\cap S^k_0$. For this case we have
\begin{align}
\mathcal{\wt{L}}_{k} f(x)&= 2^{2k}\sum_{y\leftrightarrow x \text{ in }G^k_0 }\Bigg[\sum_{i=1}^2 \frac{\partial f(x)}{\partial x_i}\cdot (y_i-x_i)+\frac{1}{2}\sum_{i, j=1}^2\frac{\partial^2 f(x)}{\partial x_i\partial x_j}\cdot (y_i-x_i)(y_j-x_j) \nonumber
\\
&\qquad +O(1)|y-x|^3\Bigg]\wt{j}_{k}(x, y)\nonumber
\\
&=2^{2k}\sum_{y\leftrightarrow x \text{ in }G^k_0}\left[\sum_{i=1}^2 \frac{\partial f(x)}{\partial x_i}\cdot (y_i-x_i)+\frac{1}{2}\sum_{i=1}^2\frac{\partial^2 f(x)}{\partial x_i^2}\cdot (y_i-x_i)^2+O(1)|y-x|^3\right]\wt{j}_{k}(x, y),\label{taylor-2}
\end{align}
where the second  $``="$ is due to the fact when $i\neq j$, either $(y_j-x_j)$ or  $(y_i-x_i)$ is zero.  It   follows  from \eqref{p_k(x,y)-1} that
 \begin{equation}\label{L_kf-1}
\mathcal{\wt{L}}_{k} f(x)=\left\{
    \begin{aligned}
        &\frac{1}{2}\Delta f(x)+O(1)2^{-k},&x\in 2^{-k}\IZ_+\cap S^k_0\\
        &\frac{1}{4}\Delta f(x)+O(1)2^{-k},&x\in D_\eps^k\cap S^k_0.
    \end{aligned}
\right.
\end{equation}
On account of   \eqref{L_kf-1} and  the fact that $\{S_0^k\}_{k\ge 1}$ is a sequence of increasing sets in $k$,  one can see that for a fixed $k_{\delta,f}\in \mathbb{N}$, as $k\rightarrow \infty$, $\mathcal{\wt{L}}_kf$ converges uniformly to 
\begin{equation}\label{def-L}
\mathcal{L}f:=\frac{1}{2}\Delta f|_{\IR_+}+\frac{1}{4}\Delta f|_{D_\eps} \quad \text{on } S_0^{k_{\delta,f}}.
\end{equation}
To prove the last statement, we recall that $f\in \mathcal{G}$ has compact and is of $C^3$-class except at the origin.  Therefore in view of \eqref{L_kf-1}, there exists $c_1>0$ such that 
\begin{equation*}
\wt{\mathcal{L}}_k f(x)\le c_1, \quad \text{for all }k\ge 1,\, \text{ all }x\in S^k_0.
\end{equation*}
It remains to show that there exists some $c_2>0$ such that 
\begin{equation}\label{e:2.36}
\wt{\mathcal{L}}_k f(x)\le c_2, \quad \text{for all }k\ge 1,\, \text{ all }x\in E^k_0\backslash S^k_0.
\end{equation}
We note that for $x\in E^k_0\backslash S_0^k$, there are three cases:
\\
{\it Case 1. }  $x=a^*_k$;
\\
{\it Case 2. } $x\in \partial E_0^k$;
\\
{\it Case 3. }  $x\in D^k_\eps$,  $x\leftrightarrow a^*_k$ in $G^k_0$.
\\
 For $f\in \mathcal{G}$, $f$ is a constant outside of  a compact subset of $(\overline{E}_0\backslash \{a^*\})\cup B_\eps$, and $f$ is a constant on the closed disc $B_\eps$. This implies that the first order derivatives of $f$ vanish outside of a compact subset of $E_0\backslash\{a^*\}$. Since $f|_{\IR^2}\in C^3(\IR^2)$  and $f|_{\IR_+}\in C^3(\IR_+)$,  $f$ has bounded second and third order derivatives.  This  implies that there exists some $c_3>0$ independent of $k$ such that
 \begin{equation*}
  |f(x)|\le c_3 \cdot 2^{-2k}, \quad \text{ if } x\in \partial E_0^k.
 \end{equation*}
 and
 \begin{equation*}
 |f(x)-f(a^*)|\le c_3 \cdot 2^{-2k}, \quad \text{ if } x\leftrightarrow a^*_k \text{ in }G^k_0.
 \end{equation*}
In view of the definition of $\wt{\mathcal{L}}_k$ in  \eqref{def-wt-Lk}, we know  there exists some constant $c_4>0$ independent of $k$ such that
\begin{equation*}
\wt{\mathcal{L}}_kf(x)\le c_4\quad \text{ on }  E^k_0\backslash S^k_0.
\end{equation*}
Finally, in view of  \eqref{e:2.36}, the proof is complete.  
\end{proof}

We prepare the following lemma for the key ingredient of the main result of this paper, Proposition \ref{P:3.8}.

\begin{lem}\label{integral-skorohkod-continuous}
For any $f\in \mathcal{G}$, any $0<t_1<t_2<T$, the following map from $\mathbf{D}([0, T], E, \rho)$ to $\IR^1$ is continuous with respect to Skorokhod topology:
\begin{equation*}
\omega\mapsto \int_{t_1}^{t_2} \left(\mathcal{L}f\right)  (\omega(s))ds.
\end{equation*}
\end{lem}
\begin{proof}
We denote by  $d$ the Skorokhod metric on $\mathbf{D}([0, T], E, \rho)$. For any $\epsilon>0$, any $\omega_1, \omega_2\in \mathbf{D}([0, T], E, \rho) $ such that $d(\omega_1, \omega_2)<\epsilon$, this means
\begin{equation*}
\inf_{\lambda \in\Lambda}\left\{\|\omega_1(t)-\omega_2(t)\|_\infty+\|\lambda(t)-t\|_\infty \right)\}<\epsilon,
\end{equation*}
where $\Lambda$ is the collection of strictly increasing continuous maps from $[0, T]$ to itself.  Thus by the fact that for any $f\in \mathcal{G}$, $\mathcal{L}f$ is bounded continuous, there exsits some $c_1>0$ such that 
\begin{equation*}
\inf_{\lambda \in\Lambda}\left\{\|\left(\mathcal{L}f\right)(\omega_1(t))- \left(\mathcal{L}f\right)(\omega_2(t))\|_\infty+\|\lambda(t)-t\|_\infty \right)\}<c_1\epsilon.
\end{equation*}
It follows that there exists a time change $\lambda_0\in \Lambda$ such that $\|\lambda_0(t)-t \|_\infty<\epsilon$, and 
\begin{equation*}
\left(\mathcal{L}f\right)\omega_2\circ \lambda_0-\epsilon \le \left(\mathcal{L}f\right)\omega_1 \le \left(\mathcal{L}f\right)\omega_2\circ \lambda_0+\epsilon\quad \text{on } [0, T]. 
\end{equation*}
Using notation
\begin{equation*}
\left(\mathcal{L}f\right)\omega^*(t, \epsilon):= \sup_{s: |s-t|<\epsilon} \left(\mathcal{L}f\right)(\omega(s)), \quad \left(\mathcal{L}f\right)\omega_*(t, \epsilon):= \inf_{s: |s-t|<\epsilon} \left(\mathcal{L}f\right)(\omega(s)), 
\end{equation*}
we get
\begin{equation*}
\left( \mathcal{L}f\right)\omega_{2*}(t, \epsilon)-\epsilon \le \mathcal{L}f\omega_1(t)\le \left( \mathcal{L}f\right)\omega_{2}^*(t, \epsilon)+\epsilon \quad \text{on }[0, T].
\end{equation*}
Both $\left( \mathcal{L}f\right)\omega_{2*}(t, \epsilon)-\epsilon$ and $\left( \mathcal{L}f\right)\omega_{2}^*(t, \epsilon)+\epsilon$ converge in $L^1[0, T]$ to $(\mathcal{L}f)\omega_1$, as $\epsilon\downarrow 0$, which implies that as $d(\omega_1, \omega_2)\downarrow 0$, $(\mathcal{L}f)\omega_1\overset{L^1[0, T]}{\longrightarrow} (\mathcal{L}f)\omega_2 $. The proof is complete.
\end{proof}

The next proposition we present is the key ingredient of the main theorem of this paper. We denote the part process of BMVD  with parameter $\eps$ characterized in Theorem \ref{BMVD-non-drift}  killed upon leaving $E_0$ by $\wh{X}$. Thus $\wh{X}$ can be characterized by the following Dirichlet form $(\wh{\EE}, \mathcal{D}(\wh{\EE}))$ on $L^2(E_0, m)$:
\begin{align}\label{killed-bmvd}
\left\{
\begin{aligned}
&{\mathcal{D}}(\wh{\EE})  = \{ f \in \mathcal{D}(\EE),\,\wt{f}=0\,\, \EE\emph{-}q.e.\text{ on } E\backslash E_0\;\},
\\
&\wh{\EE}(f,g) = \frac{1}{4} \int_{D_\eps}\nabla f(x) \cdot \nabla g(x) dx+\frac{1}{2}\int_{\IR_+}f'(x)g'(x)dx,
\end{aligned}
\right.
\end{align}
where $(\EE, \mathcal{D}(\EE))$ is given in Theorem \ref{BMVD-non-drift}, and  $\wt{f}$ is an arbitrary $\EE$-quasi-continuous $m$-version of $f$.

\begin{prop}\label{P:3.8}
For every $T>0$ and  every open bounded domain $E_0$ satisfying the assumption \ref{assumption},  let $\wt{X}$ be the limit of any   weakly convergent subsequence of $\{\wt{X}^k,\, k\ge 1\}$ with initial distribution $\overline{m}_k$. $\wt{X}$ is a continuous process on state space $\overline{E}_0$ solving the martingale problem $(\mathcal{L}, \mathcal{G})$ with initial distribution $m$. Furthermore,  the law of $\{\wt{X}_t, \,t<\tau_{E_0}\}$ coincides with $\wh{X}$ characterized by \eqref{killed-bmvd}.
\end{prop}

\begin{proof}

Since the laws of $(\wt{X}^k)_{k\ge 1}$ are C-tight in $\mathbf{D}([0, T], \overline{E}_0, \rho)$, any sequence has a weakly convergent subsequence which is supported on the set of continuous paths. Denote by $\{\wt{X}^{k_j}: j\ge 1\}$ any such  weakly convergent subsequence, and denote by $\wt{X}$ its weak limit which is  continuous. By  Skorokhod representation theorem (see, e.g.,  \cite[Chapter 3, Theorem 1.8]{EK}), we may assume that $\{\wt{X}^{k_j},j\ge 1\}$ as well as $\wt{X}$ are defined on a common probability space $(\Omega, \mathcal{F}, \IP)$, so that  $\{\wt{X}^{k_j},j\ge 1\}$  converges almost surely to $\wt{X}$ in the Skorokhod topology.

For every $t\in [0, T]$, we let $\mathcal{M}_t^{k_j} :=\sigma(\wt{X}^{k_j}_s, s\le t)$, and $\mathcal{M}_t :=\sigma(\wt{X}_s, s\le t)$. It is obvious that $\mathcal{M}_t\subset \sigma\{\mathcal{M}_t^{k_j}: j\ge 1\}$. With the class of functions $\mathcal{G}$ defined in \eqref{def-class-G}, in the following we first show that $\wt{X}$ is a solution to the martingale problem $(\mathcal{L}, \mathcal{G})$ with initial distribution $m$ with respect to the filtration $\{\mathcal{M}_t\}_{t\ge 0}$. That is,  for  every $f\in \mathcal{G}$, we need to show that 
\begin{equation*}
\left\{f(\wt{X}_t)- f(\wt{X}_0)-\int_0^t \mathcal{L}f(\wt{X}_s)ds\right\}_{t\ge 0}
\end{equation*}
is a martingale with respect to $\{\mathcal{M}_t\}_{t\ge 0}$.  For $k\ge 1$, we again denote by
\begin{equation*}
T^k_0=0, \quad \text{and }T^k_l=\inf\{t> T^k_{l-1}:\, \wt{X}^k_{T_l}\neq \wt{X}^k_{T_l-}\} \quad \text{for }l=1,2,\dots,
\end{equation*}
i.e., $T^k_l$ is the $l^{\text{th}}$ holding time of $\wt{X}^k$. $\{T^k_l:l\ge 1\}$ are i.i.d. random variables, each being  exponentially distributed with mean $2^{-2k}$. 
By \cite[Corollary 5.4.1]{FOT}, we know that for any $k\ge 1$  and  any  $f\in \mathcal{G}$,
\begin{equation*}
\left\{f(\wt{X}^k_t)-f(\wt{X}_0^k)-\int_0^t\mathcal{\wt{L}}_kf(\wt{X}^k_s)ds\right\}_{t\ge 0}
\end{equation*}
is a martingale with respect to $\{\mathcal{M}^k_t\}_{t\ge 0}$.  Therefore, for any $0\le t_1<t_2\le T$ and any $A\in \mathcal{M}_{t_1}^{k_j}$, 
it holds
\begin{equation}\label{e:3.20}
\IE^{\overline{m}_{k_j}}\left[\left(f(\wt{X}^{k_j}_{t_2})-f(\wt{X}^{k_j}_{t_1})-\sum_{l:\; t_1<T_l^{k_j}\le t_2 }\mathcal{\wt{L}}_{k_j} f\left(\wt{X}^{k_j}_{T^{k_j}_l}\right)\left(T^{k_j}_{l}- T^{k_j}_{l-1}\right)\right)\mathbf{1}_{A}\right]=0.
\end{equation}
In the following, we first claim that for any $A\in \mathcal{M}_{t_1}$,
\begin{equation}\label{e:3.21}
\IE^{\overline{m}_{k_j}}\left[\left(f(\wt{X}^{k_j}_{t_2})-f(\wt{X}^{k_j}_{t_1})\right)\mathbf{1}_{A}\right]\rightarrow \IE^{m}\left[\left(f(\wt{X}_{t_2})-f(\wt{X}_{t_1})\right)\mathbf{1}_{A}\right], \quad \text{as }j\rightarrow \infty.
\end{equation}
To prove \eqref{e:3.21},  we first note that it has been proved in Lemma \ref{L:3.4} that the laws of $\{\wt{X}^k, \IP^{\overline{m}_k}, k\ge 1\}$ are C-tight in the space $\mathbf{D}([0, T], \overline{E}_0, \rho)$ equipped with the Skorokhod topology, therefore $\wt{X}$ as a subsequential limit is a continuous process.  Also note that it has been claimed at the beginning of this proof  that  one can assume $\{\wt{X}^{k_j},j\ge 1\}$ as well as $\wt{X}$ are defined on a common probability space $(\Omega, \mathcal{F}, \IP)$, so that   $\{\wt{X}^{k_j},j\ge 1\}$  converges almost surely to $\wt{X}$ in the Skorokhod topology. From the proof of \cite[Chapter 3, Theorem 7.8]{EK}), we can tell that: Given a squence  $\{\omega_n, n\ge 1\}$ convergent to $\omega$   in the Skorokhod topology and $\omega$ is continuous at some $t_0>0$, then $\lim_{n\rightarrow \infty}\omega_n(t_0) = \omega (t_0)$. This yields that outside of a zero probability subset of $(\Omega, \FF)$ with respect to $\IP$, $\wt{X}^{k_j}_t\to \wt{X}_t$ as $j\to \infty$ for all $t\in [0, T]$. Thus  for any   positive  continuous $\Phi: \overline{E}_0\rightarrow \IR$ such that $\|\Phi\|_\infty\le 1$, it follows from  dominated convergence theorem that
\begin{equation}\label{compute-P2.10-2-1}
 \lim_{j\rightarrow \infty}\IE^{\overline{m}_{k_j}}\left[\left(f(\wt{X}^{k_j}_{t_2})-f(\wt{X}^{k_j}_{t_1}) \right)\Phi(\wt{X}_{t_1})  \right]= \IE^{m}\left[\left(f(\wt{X}_{t_2})-f(\wt{X}_{t_1}) \right)\Phi(\wt{X}_{t_1})  \right].
\end{equation}
This implies \eqref{e:3.21}.   For the summation  term in \eqref{e:3.20},  we use  a  ``$3\eps$"-argument:
\begin{eqnarray}
&&\left|\IE^{\overline{m}_{k_j}}\left[\left(\sum_{l:\; t_1<T^{k_j}_l\le t_2 }\mathcal{\wt{L}}_{k_j} f\left(\wt{X}^{k_j}_{T^{k_j}_l}\right)\right)\left(T^{k_j}_l - T^{k_j}_{l-1}\right)\mathbf{1}_{A}\right] - \IE^m\left[\left(\int_{t_1}^{t_2} \mathcal{L} f(\wt{X}_s)ds\right)\mathbf{1}_A\right]\right| \nonumber
\\
&\le &\left| \IE^{\overline{m}_{k_j}}\left[\left(\sum_{l:\; t_1<T^{k_j}_1\le t_2 }\mathcal{\wt{L}}_{k_j} f\left(\wt{X}^{k_j}_{T^{k_j}_l}\right)\left(T^{k_j}_l - T^{k_j}_{l-1}\right)-\int_{t_1}^{t_2} \mathcal{\wt{L}}_{k_j} f (\wt{X}^{k_j}_s)ds\right)\mathbf{1}_{A}\right] \right|\nonumber
\\
&+& \left|\IE^{\overline{m}_{k_j}}\left[\left(\int_{t_1}^{t_2} \mathcal{\wt{L}}_{k_j} f (\wt{X}^{k_j}_s)ds-\int_{t_1}^{t_2} \mathcal{L} f(\wt{X}^{k_j}_s)ds\right)\mathbf{1}_{A}\right] \right|\nonumber
\\
&+& \left|\IE^{\overline{m}_{k_j}}\left[\left(\int_{t_1}^{t_2} \mathcal{L} f(\wt{X}^{k_j}_s)ds\right)\mathbf{1}_{A}\right] - \IE^m\left[\left(\int_{t_1}^{t_2} \mathcal{L} f(\wt{X}_s)ds\right)\mathbf{1}_{A}\right] \right| \nonumber
\\
&=&(I)+(II)+(III).\label{e:3.27}
\end{eqnarray}
Next we claim each of the  three terms on the right hand side of \eqref{e:3.27} goes to zero as $j\rightarrow \infty$.
For the (I) on the right hand side of \eqref{e:3.27}, noticing that each $T^{k_j}_l$ is exponentially distributed at rate $2^{-2k_j}$, from Lemma \ref{Lemma2.6} we have
\begin{eqnarray}
(I)&=&\left| \IE^{\overline{m}_{k_j}}\left[\left(\sum_{l:\; t_1<T^{k_j}_l\le t_2 }\mathcal{\wt{L}}_{k_j} f\left(\wt{X}^{k_j}_{T^{k_j}_l}\right)\left(T^{k_j}_l - T^{k_j}_{l-1}\right)-\int_{t_1}^{t_2} \mathcal{\wt{L}}_{k_j} f (\wt{X}^{k_j}_s)ds\right)\mathbf{1}_{A}\right] \right|\nonumber
\\
&\le & \|\mathcal{\wt{L}}_{k_j}f\|_\infty\cdot  \IE^{\overline{m}_{k_j}}\left[\left|t_1-\sup_l\{T_l^{k_j}<t_1\}\right|+\left|\inf_l\{T_l^{k_j}>t_2\}-t_2\right|\right] \nonumber
\\
&\le & \|\mathcal{\wt{L}}_{k_j}f\|_\infty\cdot 2\cdot 2^{-2k_j} =O(1)2^{-2k_j}\rightarrow 0, \quad \text{as }j\rightarrow \infty. \label{e:3.28}
\end{eqnarray}
For (III) on the right hand side of \eqref{e:3.27},  on account of Lemma  \ref{integral-skorohkod-continuous}, we have
\begin{eqnarray} \label{e:3.30}
&& \lim_{j\rightarrow \infty}\left|\IE^{\overline{m}_{k_j}}\left[\left(\int_{t_1}^{t_2} \mathcal{L} f(\wt{X}^{k_j}_s)ds\right)\mathbf{1}_{A}\right] - \IE^{m}\left[\left(\int_{t_1}^{t_2} \mathcal{L} f(\wt{X}_s)ds\right)\mathbf{1}_{A}\right] \right| =0.
\end{eqnarray}
Finally  to show the convergence of  (II) on the right hand side of \eqref{e:3.27}, for any $\delta>0$, we let $k_{\delta, f}$ be chosen as in Lemma \ref{Lemma2.6}. Thus
\begin{eqnarray}
&&\left|\IE^{\overline{m}_{k_j}}\left[\left(\int_{t_1}^{t_2} \mathcal{\wt{L}}_{k_j} f (\wt{X}^{k_j}_s)ds-\int_{t_1}^{t_2} \mathcal{L} f(\wt{X}^{k_j}_s)ds\right)\mathbf{1}_{A}\right] \right| \nonumber
\\
&\le &\left|\IE^{\overline{m}_{k_j}}\left[\int_{t_1}^{t_2} \mathcal{\wt{L}}_{k_j} f (\wt{X}^{k_j}_s)ds-\int_{t_1}^{t_2} \mathcal{L} f(\wt{X}^{k_j}_s)ds\right] \right|\nonumber
\\
&=&\Bigg|\IE^{\overline{m}_{k_j}}\bigg[\int_{t_1}^{t_2} \mathcal{\wt{L}}_{k_j} f (\wt{X}^{k_j}_s)\mathbf{1}_{\left\{\wt{X}^{k_j}_s\in S_0^{k_{\delta, f}}\right\}}ds +\int_{t_1}^{t_2} \mathcal{\wt{L}}_{k_j} f (\wt{X}^{k_j}_s)\mathbf{1}_{\left\{\wt{X}^{k_j}_s\notin S_0^{k_{\delta, f}}\right\}}ds \nonumber
\\
&-&\int_{t_1}^{t_2} \mathcal{L} f(\wt{X}^{k_j}_s)\mathbf{1}_{\left\{\wt{X}^{k_j}_s\in S_0^{k_{\delta, f}}\right\}}ds - \int_{t_1}^{t_2} \mathcal{L} f(\wt{X}^{k_j}_s)\mathbf{1}_{\left\{\wt{X}^{k_j}_s\notin S_0^{k_{\delta, f}}\right\}} ds\bigg] \Bigg|\nonumber
\\
&\le & \IE^{\overline{m}_{k_j}}\Bigg[\bigg|\int_{t_1}^{t_2} \mathcal{\wt{L}}_{k_j} f (\wt{X}^{k_j}_s)\mathbf{1}_{\left\{\wt{X}^{k_j}_s\in S_0^{k_{\delta, f}}\right\}}ds +\int_{t_1}^{t_2} \mathcal{\wt{L}}_{k_j} f (\wt{X}^{k_j}_s)\mathbf{1}_{\left\{\wt{X}^{k_j}_s\notin S_0^{k_{\delta, f}}\right\}}ds \nonumber
\\
&-&\int_{t_1}^{t_2} \mathcal{L} f(\wt{X}^{k_j}_s)\mathbf{1}_{\left\{\wt{X}^{k_j}_s\in S_0^{k_{\delta, f}}\right\}}ds - \int_{t_1}^{t_2} \mathcal{L} f(\wt{X}^{k_j}_s)\mathbf{1}_{\left\{\wt{X}^{k_j}_s\notin S_0^{k_{\delta, f}}\right\}} ds\bigg|\Bigg]\nonumber
\\
&\le & \IE^{\overline{m}_{k_j}}\Bigg[\bigg|\int_{t_1}^{t_2} \mathcal{\wt{L}}_{k_j} f (\wt{X}^{k_j}_s)\mathbf{1}_{\left\{\wt{X}^{k_j}_s\in S_0^{k_{\delta, f}}\right\}}ds -\int_{t_1}^{t_2} \mathcal{L} f(\wt{X}^{k_j}_s)\mathbf{1}_{\left\{\wt{X}^{k_j}_s\in S_0^{k_{\delta, f}}\right\}})ds \bigg|\Bigg] \nonumber
\\
&+& \IE^{\overline{m}_{k_j}}\Bigg[ \bigg|\int_{t_1}^{t_2} \mathcal{\wt{L}}_{k_j} f (\wt{X}^{k_j}_s)\mathbf{1}_{\left\{\wt{X}^{k_j}_s\notin S_0^{k_{\delta, f}}\right\}}ds-\int_{t_1}^{t_2} \mathcal{L} f(\wt{X}^{k_j}_s)\mathbf{1}_{\left\{\wt{X}^{k_j}_s\notin S_0^{k_{\delta, f}}\right\}} ds\bigg| \Bigg].\label{e:2.40}
\end{eqnarray}
For the first term on the right hand side of \eqref{e:2.40}, by Lemma \ref{Lemma2.6} (ii),  we have
\begin{eqnarray}
&& \IE^{\overline{m}_{k_j}}\Bigg[\bigg|\int_{t_1}^{t_2} \mathcal{\wt{L}}_{k_j} f (\wt{X}^{k_j}_s)\mathbf{1}_{\left\{\wt{X}^{k_j}_s\in S_0^{k_{\delta, f}}\right\}}ds -\int_{t_1}^{t_2} \mathcal{L} f(\wt{X}^{k_j}_s)\mathbf{1}_{\left\{\wt{X}^{k_j}_s\in S_0^{k_{\delta, f}}\right\}}  ds \bigg| \Bigg] \nonumber
 \\
 &\le & \IE^{\overline{m}_{k_j}}\Bigg[\int_{t_1}^{t_2} \left|\mathcal{\wt{L}}_{k_j} f (\wt{X}^{k_j}_s)- \mathcal{L} f(\wt{X}^{k_j}_s)\right| \mathbf{1}_{\left\{\wt{X}^{k_j}_s\in S_0^{k_{\delta, f}}\right\}}ds\Bigg] \stackrel{j\rightarrow \infty}{\rightarrow} 0. \label{e:2.42}
\end{eqnarray}
For the second term on the right hand side of \eqref{e:2.40}, we first  note that $\overline{m}_{k_j}$ is an invariant measure for $\wt{X}^{k_j}$. Therefore,  
\begin{equation}\label{e:2.45}
 \IP^{\overline{m}_{k_j}}\left[\wt{X}^{k_j}_s\notin S_0^{k_{\delta, f}}\right] = \overline{m}_{k_j}\left( E^k_0\backslash S^{k_{\delta, f}}_0 \right)\stackrel{\text{Lemma \ref{Lemma2.6} (i)}}{\le } \delta
\end{equation}
Given any $\delta>0$,   for  $k_j\ge k_{\delta, f} $ where $k_{\delta, f}$ is given in Lemma \ref{Lemma2.6}, it holds
\begin{eqnarray}
&&\Bigg|\IE^{\overline{m}_{k_j}}\bigg[\int_{t_1}^{t_2} \mathcal{\wt{L}}_{k_j} f (\wt{X}^{k_j}_s)\mathbf{1}_{\left\{\wt{X}^{k_j}_s\notin S_0^{k_{\delta, f}}\right\}}ds+\int_{t_1}^{t_2} \mathcal{L} f(\wt{X}^{k_j}_s)\mathbf{1}_{\left\{\wt{X}^{k_j}_s\notin S_0^{k_{\delta, f}}\right\}} ds\bigg] \Bigg| \nonumber
\\
&\le & \IE^{\overline{m}_{k_j}}\Bigg[\bigg|\int_{t_1}^{t_2} \mathcal{\wt{L}}_{k_j} f (\wt{X}^{k_j}_s)\mathbf{1}_{\left\{\wt{X}^{k_j}_s\notin S_0^{k_{\delta, f}}\right\}}ds+\int_{t_1}^{t_2} \mathcal{L} f(\wt{X}^{k_j}_s)\mathbf{1}_{\left\{\wt{X}^{k_j}_s\notin S_0^{k_{\delta, f}}\right\}} ds\bigg|\Bigg] \nonumber
\\
&\le &\IE^{\overline{m}_{k_j}}\Bigg[\int_{t_1}^{t_2} \left|\mathcal{\wt{L}}_{k_j} f (\wt{X}^{k_j}_s)+ \mathcal{L} f(\wt{X}^{k_j}_s)\right| \mathbf{1}_{\left\{\wt{X}^{k_j}_s\notin S_0^{k_{\delta, f}}\right\}}ds\Bigg] \nonumber
\\
\text{(Lemma \ref{Lemma2.6})}&\le &\left( C_1+\|\mathcal{L} f \|_\infty \right)\int^{t_2}_{t_1}  \IP^{\overline{m}_{k_j}}\left[\wt{X}^{k_j}_s\notin S_0^{k_{\delta, f}}\right]   ds \nonumber
\\
&\stackrel{\eqref{e:2.45}}{\le} &  \left( C_1+\|\mathcal{L} f \|_\infty \right) (t_2-t_1) \delta,\nonumber
\end{eqnarray}
where $C_1$ is the same as in Lemma \ref{Lemma2.6}.  This shows that 
\begin{equation}\label{e:2.43}
\Bigg|\IE^{\overline{m}_{k_j}}\bigg[\int_{t_1}^{t_2} \mathcal{\wt{L}}_{k_j} f (\wt{X}^{k_j}_s)\mathbf{1}_{\left\{\wt{X}^{k_j}_s\notin S_0^{k_{\delta, f}}\right\}}ds+\int_{t_1}^{t_2} \mathcal{L} f(\wt{X}^{k_j}_s)\mathbf{1}_{\left\{\wt{X}^{k_j}_s\notin S_0^{k_{\delta, f}}\right\}} ds\bigg] \Bigg|\stackrel{j\rightarrow \infty}{\rightarrow} 0.
\end{equation}
Combining \eqref{e:2.42} and \eqref{e:2.43} shows that the right hand side of \eqref{e:2.40}, i.e., (II) on the right hand side of \eqref{e:3.27} goes to zero as $j\rightarrow \infty$. This again combined with \eqref{e:3.28} and \eqref{e:3.30} verifies that all three terms, (I)-(III) on the right hand side of \eqref{e:3.27} tend to zero as $j\rightarrow \infty$.   This combined with \eqref{e:3.21} establishes \eqref{e:3.20} for any $f\in \mathcal{G}$ and any $A\in \mathcal{M}^{k_j}_{t_1} $, given any $0\le t_1<t_2\le T$. In view of the fact that 
$\mathcal{M}_t\subset \sigma\{\mathcal{M}_t^{k_j}:\, j\ge 1\}$ for all $t\ge 0$, \eqref{e:3.20} holds for all $f\in \mathcal{G}$ and all $A\in \mathcal{M}_{t_1}$. Namely, for any $f\in \mathcal{G}$,
\begin{equation*}
\left\{f(\wt{X}_t)- f(\wt{X}_0)-\int_0^t \mathcal{L}f(\wt{X}_s)ds\right\}_{t\in [0, T]}
\end{equation*}
is a martingale with respect to  $\{\mathcal{M}_t\}_{t\ge 0}$.  This shows that $\wt{X}$ is a solution to the martingale problem $(\mathcal{L}, \mathcal{G})$. Finally, to show the last claim of this proposition, we denote  the infinitesimal generator of $\wh{X}$ by $(\wh{\mathcal{L}},\mathcal{D}(\wh{\mathcal{L}}))$.  $\mathcal{D}(\wh{\mathcal{L}})$ can be described as follows: 
$u\in \mathcal{D}(\wh{\mathcal{L}})$   if and only if: $u\in \mathcal{D}(\wh{\EE})$  given in \eqref{killed-bmvd}, and there is an $f\in L^2(E_0,m)$  satisfying
\begin{equation*}
\wh{\EE}(u, v)=-\int_{E_0} f(x)v(x)m(dx) \quad \textrm{for every }v\in \mathcal{D}(\wh{\EE}),
\end{equation*}
in which case $f=\wh{\mathcal{L}}u$.   It also holds that $\wh{\mathcal{L}}=\mathcal{L}$ on $\mathcal{D}(\wh{\mathcal{L}})$. It is mentioned in Theorem \ref{BMVD-non-drift}  that BMVD is a Feller process with strong Feller property, therefore so is its part process $\wh{X}$.  By \cite[Theorem 3.1, Remark 3.3]{MF}, this implies that the  $\mathbf{D}([0, \infty), E_0, \rho)$  martingale problem $(\wh{\mathcal{L}}, \mathcal{D}(\wh{\mathcal{L}}))$ is well-posed with its unique solution being $\wh{X}$. It is not hard to see that the bp-closures  (whose definition can be found, e.g., \cite[Definition 3.4.3]{AB})  of the graphs of  $(\mathcal{L}|_{\mathcal{G}\cap \mathcal{D}(\wh{\mathcal{L}})}, \mathcal{G}\cap \mathcal{D}(\wh{\mathcal{L}}))$ and $(\wh{\mathcal{L}}, \mathcal{D}(\wh{\mathcal{L}}))$ are the same. Hence by \cite[Proposition 3.4.19]{AB} the  $\mathbf{D}([0, \infty), E_0, \rho)$ martingale problem $(\mathcal{L}|_{\mathcal{G}\cap \mathcal{D}(\wh{\mathcal{L}})}, \mathcal{G}\cap \mathcal{D}(\wh{\mathcal{L}}))$ is also well-posed with the unique solution $\wh{X}$.

 $\wt{X}$ is a solution to the  martingale problem $(\mathcal{L}, \mathcal{G})$, therefore $\{\wt{X}_t, t<\tau_{E_0}, \IP^m\}$  must be  a solution to the $\mathbf{D}([0, \infty), E_0, \rho)$   martingale problem $(\mathcal{L}|_{\mathcal{G}\cap \mathcal{D}(\wh{\mathcal{L}})}, \mathcal{G}\cap \mathcal{D}(\wh{\mathcal{L}}))$.  Since we have  claimed at the end of last paragraph that  the  $\mathbf{D}([0, \infty), E_0, \rho)$ martingale problem $(\mathcal{L}|_{\mathcal{G}\cap \mathcal{D}(\wh{\mathcal{L}})}, \mathcal{G}\cap \mathcal{D}(\wh{\mathcal{L}}))$ is well-posed with the unique solution being $\wh{X}$, the proof is complete. 
\end{proof}

We prepare the following lemma before proving the main theorem. Let $E_0^\Delta$ be the one point compactification of $E_0$, where $\Delta$ stands for the cemetry point. The metric space $(E_0, \rho)$ can be extended to a complete and separable metric space $(E_0^\Delta, \rho)$ naturally as follows: 
\begin{equation}
\rho(x, \Delta):=\rho(x, \partial E_0), \quad \text{for }x\in E_0.
\end{equation}
\begin{lem}\label{L:2.10}
The laws of $\{\wh{X}^{k}, \IP^{\overline{m}_{k}}\}$ are also C-tight in $\mathbf{D}([0, T], E_0^\Delta,\rho)$.
\end{lem}
\begin{proof}
First of all, in view of  \eqref{DF-RWVD-form} and the fact that $\{\wh{X}^k, k\ge 1\}$ is the part process of $\{X^k, k\ge 1\}$, we note that $\{\wh{X}^k, k\ge 1\}$ can be characterized by the following regular symmetric Dirichlet form on $L^2(E^k, m_k)$:
\begin{align}\label{DF-killed}
\left\{
\begin{aligned}
&\mathcal{D}(\wh{\EE}^{k})=\left\{f: f\in  \mathcal{D}(\EE^k),\, \text{supp}[f]\subset E_0^k\right\},
\\
&\wh{\EE}^{k}(f, f)= \frac{1}{8}\sum_{\substack{e^o_{xy}:\; e_{xy}\in G^k_e,\\ x,y\in D^k_\eps\cup \{a_k^*\} }} \left(f(x)-f(y)\right)^2 +\frac{2^k}{4}\sum_{\substack{e^o_{xy}:\;e_{xy}\in G^k_e,\\ x,y\in 2^{-k}\IZ_+\cup \{a_k^*\} }}\left(f(x)-f(y)\right)^2.
\end{aligned}
\right.
\end{align}
Recall that $v_k$ and $\bar{v}_k$ are both defined  in Section 1, along with \eqref{def-mk} and \eqref{def-bar-mk}.   Comparing the \eqref{DF-killed} with $(\mathcal{D}(\wt{\EE}^k), \wt{\EE}^k))$ given in \eqref{reflected-form}, we see that by letting 
\begin{equation*}
\kappa(dx):=\frac{1}{4}\left(v_k(x)-\bar{v}_k(x)\right)\mathbf{1}_{\{x\in D^k_\eps\}}+\frac{2^k}{2}\left(v_k(x)-\bar{v}_k(x)\right)\mathbf{1}_{\{x\in 2^{-k}\IZ_+\}}, \quad \text{for }x\in E^k_0,
\end{equation*}
one can then  write  for $f\in \mathcal{D}(\wh{\EE}^{k})$  vanishing outside of $E_0^k$ that
\begin{eqnarray}\label{beurling-deny-whXk}
\wh{\EE}^k(f, f)&=& \frac{1}{8}\sum_{\substack{e^o_{xy}:\; e_{xy}\in G^k_e,\\ x,y\in D^k_\eps\cup \{a_k^*\} }} \left(f(x)-f(y)\right)^2 +\frac{2^k}{4}\sum_{\substack{e^o_{xy}:\;e_{xy}\in G^k_e,\\ x,y\in 2^{-k}\IZ_+\cup \{a_k^*\} }}\left(f(x)-f(y)\right)^2 \nonumber
\\
&=& \frac{1}{8}\sum_{\substack{e^o_{xy}:\;e_{xy}\in G^k_{0, e},\\ x,y\in D^k_\eps\cup \{a_k^*\} }}\left(f(x)-f(y)\right)^2 +\frac{2^k}{4}\sum_{\substack{e^o_{xy}:\;e_{xy}\in G^k_{0, e},\\ x,y\in 2^{-k}\IZ_+\cup \{a_k^*\} }}\left(f(x)-f(y)\right)^2 \nonumber
\\
&+&  \frac{1}{8}\sum_{\substack{e^o_{xy}:\;e_{xy}\in (G^k_e\backslash  G^k_{0, e}),\\ x,y\in D^k_\eps\cup \{a_k^*\} }}\left(f(x)-f(y)\right)^2 +\frac{2^k}{4}\sum_{\substack{e^o_{xy}:\;e_{xy}\in (G^k_e\backslash  G^k_{0, e}),\\ x,y\in 2^{-k}\IZ_+\cup \{a_k^*\} }}\left(f(x)-f(y)\right)^2 \nonumber
\\
&=& \wt{\EE}^k(f, f)+\frac{1}{4}\sum_{\substack{x,y\in D^k_\eps\cup \{a_k^*\}\\ x\in E_0^k, \,y\notin E_0^k, \\x\leftrightarrow y \text{ in } G^k}}f(x)^2 +\frac{2^k}{2}\sum_{\substack{ x,y\in 2^{-k}\IZ_+\cup \{a_k^*\} \\ x\in E_0^k,\, y\notin E_0^k, \\x\leftrightarrow y \text{ in } G^k }}f(x)^2 \nonumber
\\
& =&\wt{\EE}^k(f, f)+\int_{E_0^k} f(x)^2\kappa(dx).
\end{eqnarray}
Since $\wt{X}^k$ is a pure jump process without killing, \eqref{beurling-deny-whXk} can be viewed as the Beurling-Deny decomposition of $(\mathcal{D}(\wt{\EE}^k), \wt{\EE}^k))$. Also since $E^k_0$ is a bounded space containing only finitely many isolated points, by \cite[Theorem 5.2.17]{CF}, it is clear that $(\mathcal{D}(\wt{\EE}^k), \wt{\EE}^k))$ is a resurrected Dirichlet form of $(\mathcal{D}(\wh{\EE}^k), \wh{\EE}^k))$, and $\wt{X}^k$ is a resurrected process of $\wh{X}^k$, for every $k\in \mathbb{N}$. By the same argument as in the proof to \cite[Theorem 2.1]{BBC}, it can be proved that $\wt{X}^k$ can be equivalently constructed from $\wh{X}^k$ through the Ikeda–Nagasawa–Watanabe ``piecing together" procedure developed in \cite{INW}: Let $\tau_{E^k_0}^k$ be the lifetime of $\wh{X}^k$. Let $\wt{X}^k_t(\omega)=\wh{X}^k_t(\omega)$ for $t<\tau_{E^k_0}^k$. Let $\wt{X}^k_{\tau_{E^k_0}^k}(\omega)=\wh{X}^k_{\tau_{E^k_0}^k-}(\omega)$ and glue an independent copy of $\wh{X}^k$ starting from $\wh{X}^k_{\tau_{E^k_0}^k-}(\omega)$ to $\wt{X}^k_{\tau_{E^k_0}^k}(\omega)$. Iterate this procedure countably many times. 

Recall the definition of ``$w_\rho$" given in \eqref{def-dubiu}.   From the Ikeda-Nagasawa-Watanabe `` piecing together" procedure, one can tell that  for any  $\omega\in \Omega$ and any $k\in \mathbb{N}$ that
\begin{eqnarray}
w_\rho(\wh{X}^k(\omega), \theta, T)\le w_\rho(\wt{X}^k(\omega), \theta, T)+\rho(\wh{X}^k_{\tau^k_{E_0^k}-}(\omega), \Delta) 
\le w_\rho(\wt{X}^k(\omega), \theta, T)+2^{-k}. \label{L:2.10-compute-1}
\end{eqnarray}
We slightly abuse the notation in the above inequalities by using the same ``$w_\rho$"  for both the case where the underlying metric space is $(E_0^\Delta, \rho)$ as well as when the underlying metric space is $(\overline{E}_0, \rho)$. Recall that we have established \eqref{e:3.12-1} and \eqref{e:3.13-1}  in the proof to Lemma \ref{L:3.4}.  Hence by \eqref{L:2.10-compute-1}, it holds  for $(\wh{X}^k, \IP^{\overline{m}_k})$   that:
\begin{description}
\item{(i).} For any $T>0$, $\delta>0$, there exist $k_1\in \mathbb{N}$ and $K>0$ such that for all $k\ge k_1$,
\begin{equation*}\label{e:3.12-1-v2}
\IP^{\overline{m}_k}\left[ \sup_{t\in [0, T]} \big|\wh{X}^k_t\big|_\rho>N\right]<\delta. 
\end{equation*} 
\item{(ii).} For any $T>0$, $\delta_1, \delta_2>0$, there exist $\theta>0$ and $k_2>0$ such that for all $k\ge k_2$,
\begin{equation*}\label{e:3.13-1-v2}
\IP^{\overline{m}_k}\left[w_\rho\left( \wh{X}^k  ,\, \theta,\, T\right)>\delta_1\right]<\delta_2. 
\end{equation*}
\end{description}
Thus by  \cite[Chapter 3, Theorem 2.2 and 7.2]{EK},  the laws of $\{\wh{X}^{k}, \IP^{\overline{m}_{k}}\}$ are tight in $\mathbf{D}([0, T], E_0^\Delta,\rho)$. Finally, in view of \cite[Chapter 3, Theorem 10.2]{EK}, the laws of $\{\wh{X}^{k}, \IP^{\overline{m}_{k}}\}$ are C-tight in $\mathbf{D}([0, T], E_0^\Delta,\rho)$.
\end{proof}

The next theorem is  essentially  a restatement of our main result, Theorem \ref{main-result}. For notational convenience, for $k\ge 1$, we set
\begin{equation*}
D^k_0:=S^k_0\cup \{x: x=a^*_k\text{ or }x\leftrightarrow a^*_k\text{ in } G^k_0\},
\end{equation*}
where $S_0^k$ is defined in \eqref{def-S0k}.   $\{D^k_0\}_{k\ge 1}$ is an increasing sequence of  subsets of  $E_0$. 
\begin{thm}
$\{(\wh{X}^{k},\IP^{\overline{m}_k}),\, k\ge 1\}$ converges weakly to $(\wh{X},\,\IP^m)$.
\end{thm}
\begin{proof}

On account of Lemma \ref{L:2.10},  it remains to show that any weakly convergent subsequence of $\{(\wh{X}^{k}, \IP^{\overline{m}_k}), k\ge 1\}$ converges to $(\wh{X},\IP^m)$ characterized by \eqref{killed-bmvd}. Let $\{k_j\}_{j\ge 1}$ be a subsequence  of $\mathbb{N}$ such that $\{\wh{X}^{k_j}, \IP^{\overline{m}_{k_j}}\}$ converges weakly. By Proposition \ref{P:3.8}, $\{k_j\}_{j\ge 1}$ further  admits a subsequence $\{k_{j_l}\}_{l\ge 1}$ such that $\wt{X}^{k_{j_l}}$ converges weakly. Given any finitely many compact measurable sets $A_1, \dots, A_n \subset E_0$, for sufficiently large $k\in \mathbb{N}$, it holds
 \begin{equation}\label{e:2.47}
 A_1, \dots, A_n\subset D^k_0.
 \end{equation}
Note that $\wt{X}^k$ and $\wh{X}^k$ have the same distribution on $D^k_0$. For any $j\ge 1$, denote $\wt{\tau}^j_{D^k_0}:=\inf\{t>0: \wt{X}^j\notin D^k_0\}$, $\wt{\tau}_{D^k_0}:=\inf\{t>0: \wt{X}\notin D^k_0\}$, and $\wh{\tau}^j_{D^k_0}:=\inf\{t>0: \wh{X}^j\notin D^k_0\}$. Since  $\wh{X}^k$ is the part process of $\wt{X}^k$ killed upon hitting $\partial E_0^k$, $\wt{\tau}^j_{D^k_0}= \wh{\tau}^j_{D^k_0}$ for all $k, j\in \mathbb{N}$.    We then have  for any $0\le t_1<\cdots <t_n\le T$,
 \begin{eqnarray}
&&\lim_{j\rightarrow \infty} \IP^{\overline{m}_{k_j}}\left[\left(\wh{X}^{k_j}_{t_1}, \dots, \wh{X}^{k_j}_{t_n}\right) \in A_1\times \cdots \times A_n\right] \nonumber
\\
&\stackrel{\eqref{e:2.47}}{=}& \lim_{l\rightarrow \infty} \IP^{\overline{m}_{k_{j_l}}}\left[\left(\wh{X}^{k_{j_l}}_{t_1}, \dots, \wh{X}^{k_{j_l}}_{t_n}\right) \in A_1\times \cdots \times A_n;\, t_n<\wh{\tau}^{k_{j_l}}_{D^k_0}\right] \nonumber
\\
 &=& \lim_{l\rightarrow \infty} \IP^{\overline{m}_{k_{j_l}}}\left[\left(\wt{X}^{k_{j_l}}_{t_1}, \dots, \wt{X}^{k_{j_l}}_{t_n}\right) \in A_1\times \cdots \times A_n;\; t_n<\wt{\tau}^{k_{j_l}}_{D^k_0}\right] \nonumber
 \\
 (\text{Proposition }\ref{P:3.8})&=& \IP^m\left[\left(\wt{X}_{t_1}, \dots, \wt{X}_{t_n}\right) \in A_1\times \cdots \times A_n;\; t_n<\wt{\tau}_{D^k_0}\right] \nonumber
 \\
&=& \IP^m\left[\left(\wh{X}_{t_1}, \dots, \wh{X}_{t_n}\right) \in A_1\times \cdots \times A_n\right] \nonumber.
 \end{eqnarray}
 This verifies that $\{\wh{X}^{k_j}, \IP^{\overline{m}_k}\}$ converges to $\wh{X}$ characterized by \eqref{killed-bmvd}. Since the laws of $\{(\wh{X}^k, \IP^{\overline{m}_k}),\\ k\ge 1\}$ are tight in $\mathbf{D}([0, T], E_0^\Delta,\rho)$, and $\{(\wh{X}^{k_j}, \IP^{\overline{m}_{k_j}}),j\ge 1\}$ is an arbitrary choice of weakly convergent subsequence of $\{\wh{X}^{k}, \IP^{\overline{m}_{k}}\}$, the proof is complete. 
\end{proof}

\medskip

{\bf Acknowledgement.} I thank Professor Zhen-Qing Chen for suggesting the topic and some very helpful comments.

\vskip 0.3truein

\noindent {\bf Shuwen Lou}

\smallskip \noindent
Department of Mathematics and Statistics, Loyola University Chicago,
\noindent
Chicago, IL 60660, USA

\noindent
E-mail:  \texttt{slou1@luc.edu}

 \end{document}